\documentclass[english,11pt]{amsart}
\usepackage[utf8]{inputenc}
\usepackage{amssymb,amsmath,amsthm,mathtools,amsfonts}
\usepackage{graphicx}
\usepackage[usenames,dvipsnames]{xcolor}
\usepackage{cite}
\usepackage{url}
\usepackage{amsmath}
\usepackage{hyperref}
\usepackage{pgfplots}\pgfplotsset{compat=1.13}
\usepackage{xcolor}
\usepackage[all]{xy}
\usepackage[version=4]{mhchem}
\usepackage{pgfplots}
\usepgfplotslibrary{polar}
\hypersetup{
	bookmarks=true,         	
	unicode=false,          		
	pdftoolbar=true,        		
	pdfmenubar=true,        	
	pdffitwindow=false,     		
	pdfstartview={FitH},    		
	pdftitle={long-time dynamics of the sine-Gordon equation },    					
	pdfauthor={Cheng He, Jiaqi Liu, Changzheng Qu},     	
	linktocpage=true,			
	pdfnewwindow=true,      	
	colorlinks=true,       			
	linkcolor=red,          		
	citecolor=PineGreen,    	
	filecolor=magenta,      		
	urlcolor=cyan           		
}

\usepackage{tikz}
\usetikzlibrary{decorations.pathmorphing}
\usetikzlibrary{arrows,decorations.pathmorphing}
\usetikzlibrary{patterns}
\usepackage{float}
\usepackage[section]{placeins}		

\theoremstyle{plain}
\newtheorem{theorem}{Theorem}[section]

\newtheorem{lemma}[theorem]{Lemma}
\newtheorem{proposition}[theorem]{Proposition}

\theoremstyle{definition}
\newtheorem{definition}[theorem]{Definition}
\newtheorem{problem}[theorem]{Problem}

\theoremstyle{remark}
\newtheorem{remark}[theorem]{Remark}

\numberwithin{figure}{section}
\numberwithin{equation}{section}

\DeclareMathOperator{\ad}{ad}

\DeclareMathOperator{\Real}{Re}

\DeclareMathOperator{\Res}{Res}

\allowdisplaybreaks

\title[defocussing NLS]{Partial mass dynamics of the defocussing nonlinear Schr\"odinger equation}

\author{Jiaqi Liu}
\author{Xixi Xu}
\address[Liu]{School of mathematics, University of Chinese Academy of Sciences. No.19 Yuquan Road, Beijing China }
\email{jqliu@ucas.ac.cn}
\address[Xu]{School of mathematics, University of Chinese Academy of Sciences. No.19 Yuquan Road, Beijing China }
\email{18428374320@163.com}
\date{\today}

\begin{document}
	
	\maketitle
	\begin{abstract}
		We study the long time dynamics of the defocussing NLS equation. Compared with previous literature, we revisit the direct and inverse scattering map to obtain asymptotics in some weighted energy space that requires less  restrictive decay and  regularity assumptions. The main result is derived from an application of uniform resolvent bound and an approximation argument in the spirit of  \textit{Riemann-Lebesgue} lemma. As a consequence, our result depicts the long time dynamics of the zeros of the solution to the defocussing NLS equation.
	\end{abstract}
	
	\tableofcontents
	
	%
	%
	
	\newcommand{\eps}{\varepsilon}
	\newcommand{\lam}{\zeta}
	
	\newcommand{\bfN}{\mathbf{N}}
	\newcommand{\calbR}{\mathcal{ \breve{R}}}
	\newcommand{\rhobar}{\overline{\rho}}
	\newcommand{\zetabar}{\overline{\zeta}}
	
	\newcommand{\rarr}{\rightarrow}
	\newcommand{\darr}{\downarrow}
	
	\newcommand{\dee}{\partial}
	\newcommand{\dbar}{\overline{\partial}}
	
	\newcommand{\dint}{\displaystyle{\int}}
	
	\newcommand{\dotarg}{\, \cdot \, }

	%
	%
	
	\newcommand{\RHP}{\mathrm{LC}}			
	\newcommand{\PC}{\mathrm{PC}}
	\newcommand{\w}{w^{(2)}}
	%
	%
	
	\newcommand{\zbar}{\overline{z}}

	\newcommand{\bbC}{\mathbb{C}}
	\newcommand{\bbR}{\mathbb{R}}

	\newcommand{\calB}{\mathcal{B}}
	\newcommand{\calC}{\mathcal{C}}
	\newcommand{\calR}{\mathcal{R}}
	\newcommand{\calS}{\mathcal{S}}
	\newcommand{\calZ}{\mathcal{Z}}
	\newcommand{\tgamma}{\tilde{\gamma}}

	\newcommand{\ba}{\breve{a}}
	\newcommand{\bb}{\breve{b}}
	\newcommand{\bchi}{\breve{\chi}}
	
	\newcommand{\balpha}{\breve{\alpha}}
	\newcommand{\brho}{\breve{\rho}}

	\newcommand{\tPhi}{{\widetilde{\Phi}}}

	\newcommand{\tp}{\text{p}}
	\newcommand{\tq}{\text{q}}
	\newcommand{\tr}{\text{r}}
	
	\newcommand{\bfe}{\mathbf{e}}
	\newcommand{\bfn}{\mathbf{n}}
	
	\newcommand{\tA}{\tilde{A}}
	\newcommand{\tB}{\tilde{B}}
	\newcommand{\tomega}{\tilde{\omega}}
	\newcommand{\tc}{\tilde{c}}

	\newcommand{\mhat}{\hat{m}}
	
	\newcommand{\bphi}{\breve{\Phi}}
	\newcommand{\bN}{\breve{N}}
	\newcommand{\bV}{\breve{V}}
	\newcommand{\bR}{\breve{R}}
	\newcommand{\bdelta}{\breve{\delta}}
	\newcommand{\bzeta}{\breve{\zeta}}
	\newcommand{\bbeta}{\breve{\beta}}
	\newcommand{\bm}{\breve{m}}
	\newcommand{\br}{\breve{r}}
	\newcommand{\bnu}{\breve{\nu}}
	\newcommand{\bbfN}{\breve{\mathbf{N}}}
	\newcommand{\rbar}{\overline{r}}
	
	\newcommand{\One}{\mathbf{1}}
	\newcommand{\tabincell}[2]{\begin{tabular}{@{}#1@{}}#2\end{tabular}}%
	\renewcommand{\Re}{\operatorname{Re}}
	\renewcommand{\Im}{\operatorname{Im}}
	\newcommand{\diag}{\textrm{diag}}
	\newcommand{\off}{\textrm{off}}
	\newcommand{\T}{\textrm{T}}
	%
	%
	
	\newcommand{\bigO}[2][ ]
	{
		\mathcal{O}_{#1}
		\left(
		{#2}
		\right)
	}
	
	\newcommand{\littleO}[1]{{o}\left( {#1} \right)}
	
	\newcommand{\norm}[2]
	{
		\left\Vert		{#1}	\right\Vert_{#2}
	}
	
	%
	%
	\newcommand{\colvec}[2]
	{
		\left(
		\begin{array}{c}
			{#1}\\	
                {#2}	
		\end{array}
		\right)
	}
	\newcommand{\rowvec}[2]
	{
		\left(
		\begin{array}{cc}
			{#1}	&	{#2}	
		\end{array}
		\right)
	}
	
	\newcommand{\uppermat}[1]
	{
		\left(
		\begin{array}{cc}
			0		&	{#1}	\\
			0		&	0
		\end{array}
		\right)
	}
	
	\newcommand{\lowermat}[1]
	{
		\left(
		\begin{array}{cc}
			0		&	0	\\
			{#1}	&	0
		\end{array}
		\right)
	}
	
	\newcommand{\offdiagmat}[2]
	{
		\left(
		\begin{array}{cc}
			0		&	{#1}	\\
			{#2}	&	0
		\end{array}
		\right)
	}
	
	\newcommand{\diagmat}[2]
	{
		\left(
		\begin{array}{cc}
			{#1}	&	0	\\
			0		&	{#2}
		\end{array}
		\right)
	}
	
	\newcommand{\Offdiagmat}[2]
	{
		\left(
		\begin{array}{cc}
			0			&		{#1} 	\\
			\\
			{#2}		&		0
		\end{array}
		\right)
	}
	
	\newcommand{\twomat}[4]
	{
		\left(
		\begin{array}{cc}
			{#1}	&	{#2}	\\
			{#3}	&	{#4}
		\end{array}
		\right)
	}
	
	\newcommand{\unitupper}[1]
	{	
		\twomat{1}{#1}{0}{1}
	}
	
	\newcommand{\unitlower}[1]
	{
		\twomat{1}{0}{#1}{1}
	}
	
	\newcommand{\Twomat}[4]
	{
		\left(
		\begin{array}{cc}
			{#1}	&	{#2}	\\[10pt]
			{#3}	&	{#4}
		\end{array}
		\right)
	}
	
	%
	%
	%
	
	\newcommand{\JumpMatrixFactors}[6]
	{
		\begin{equation}
			\label{#2}
			{#1} =	\begin{cases}
				{#3} {#4}, 	&	\lambda \in (-\infty,\xi) \\
				\\
				{#5}{#6},	&	\lambda \in (\xi,\infty)
			\end{cases}
		\end{equation}
	}

	
	%
	%
	%
	
	\newcommand{\RMatrix}[9]
	{
		\begin{equation}
			\label{#1}
			\begin{aligned}
				\left. R_1 \right|_{(\xi,\infty)} 	&= {#2} &	\qquad\qquad		
				\left. R_1 \right|_{\Sigma_1}		&= {#3}
				\\[5pt]
				\left. R_3 \right|_{(-\infty,\xi)} 	&= {#4} 	&	
				\left. R_3 \right|_{\Sigma_2} 	&= {#5}
				\\[5pt]
				\left. R_4 \right|_{(-\infty,\xi)} 	&= {#6} &	
				\left. R_4 \right|_{\Sigma_3} 	&= {#7}
				\\[5pt]
				\left. R_6 \right|_{(\xi,\infty)}  	&= {#8} &	
				\left. R_6 \right|_{\Sigma_4} 	&= {#9}
			\end{aligned}
		\end{equation}
	}
	
	%
	%
	
	%
	%
	%
	%
	%
	%
	
	\newcommand{\SixMatrix}[6]
	{
		\begin{figure}
			\centering
			\caption{#1}
			\vskip 15pt
			\begin{tikzpicture}
				[scale=0.75]
				%
				%
				\draw[thick]	 (-4,0) -- (4,0);
				\draw[thick] 	(-4,4) -- (4,-4);
				\draw[thick] 	(-4,-4) -- (4,4);
				%
				%
				\draw	[fill]		(0,0)						circle[radius=0.075];
				\node[below] at (0,-0.1) 				{$z_0$};
				%
				%
				\node[above] at (3.5,2.5)				{$\Omega_3$};
				\node[below]  at (3.5,-2.5)			{$\Omega_7$};
				\node[above] at (0,3.25)				{$\Omega_1$};
				\node[below] at (0,-3.25)				{$\Omega_2$};
				\node[above] at (-3.5,2.5)			{$\Omega_8$};
				\node[below] at (-3.5,-2.5)			{$\Omega_4$};
				%
				%
				\node[above] at (0,1.25)				{$\twomat{1}{0}{0}{1}$};
				\node[below] at (0,-1.25)				{$\twomat{1}{0}{0}{1}$};
				%
				%
				\node[right] at (1.20,0.70)			{$#3$};
				\node[left]   at (-1.20,0.70)			{$#4$};
				\node[left]   at (-1.20,-0.70)			{$#5$};
				\node[right] at (1.20,-0.70)			{$#6$};
			\end{tikzpicture}
			\label{#2}
		\end{figure}
	}
	
	\newcommand{\sixmatrix}[6]
	{
		\begin{figure}
			\centering
			\caption{#1}
			\vskip 15pt
			\begin{tikzpicture}
				[scale=0.75]
				%
				%
				\draw[thick]	 (-4,0) -- (4,0);
				\draw[thick] 	(-4,4) -- (4,-4);
				\draw[thick] 	(-4,-4) -- (4,4);
				%
				%
				\draw	[fill]		(0,0)						circle[radius=0.075];
				\node[below] at (0,-0.1) 				{$-z_0$};
				%
				%
				\node[above] at (3.5,2.5)				{$\Omega_9$};
				\node[below]  at (3.5,-2.5)			{$\Omega_5$};
				\node[above] at (0,3.25)				{$\Omega_1$};
				\node[below] at (0,-3.25)				{$\Omega_2$};
				\node[above] at (-3.5,2.5)			{$\Omega_6$};
				\node[below] at (-3.5,-2.5)			{$\Omega_{10}$};
				%
				%
				\node[above] at (0,1.25)				{$\twomat{1}{0}{0}{1}$};
				\node[below] at (0,-1.25)				{$\twomat{1}{0}{0}{1}$};
				%
				%
				\node[right] at (1.20,0.70)			{$#3$};
				\node[left]   at (-1.20,0.70)			{$#4$};
				\node[left]   at (-1.20,-0.70)			{$#5$};
				\node[right] at (1.20,-0.70)			{$#6$};
			\end{tikzpicture}
			\label{#2}
		\end{figure}
	}
	
	\newcommand{\JumpMatrixRightCut}[6]
	{
		\begin{figure}
			\centering
			\caption{#1}
			\vskip 15pt
			\begin{tikzpicture}[scale=0.85]
				%
				%
				\draw [fill] (4,4) circle [radius=0.075];						
				\node at (4.0,3.65) {$\xi$};										
				%
				%
				\draw 	[->, thick]  	(4,4) -- (5,5) ;								
				\draw		[thick] 		(5,5) -- (6,6) ;
				\draw		[->, thick] 	(2,6) -- (3,5) ;								
				\draw		[thick]		(3,5) -- (4,4);	
				\draw		[->, thick]	(2,2) -- (3,3);								
				\draw		[thick]		(3,3) -- (4,4);
				\draw		[->,thick]	(4,4) -- (5,3);								
				\draw		[thick]  		(5,3) -- (6,2);
				%
				%
				\draw [  thick, blue, decorate, decoration={snake,amplitude=0.5mm}] (4,4)  -- (8,4);				
				\node at (1.5,4) {$0 < \arg (\zeta-\xi) < 2\pi$};
				%
				%
				\node at (8.5,8.5)  	{$\Sigma_1$};
				\node at (-0.5,8.5) 	{$\Sigma_2$};
				\node at (-0.5,-0.5)	{$\Sigma_3$};
				\node at (8.5,-0.5) 	{$\Sigma_4$};
				%
				%
				\node at (7,7) {${#3}$};						
				\node at (1,7) {${#4}$};						
				\node at (1,1) {${#5}$};						
				\node at (7,1) {${#6}$};						
			\end{tikzpicture}
			\label{#2}
		\end{figure}
	}
	
	%
	%

\section{Introduction}
In this paper we apply the inverse scattering transform (IST) to study the long time dynamics of the defocusing nonlinear Schr\"odinger (NLS) equation on the real line with finite density initial data:
\begin{equation}
 \label{eq:NLS}
\left\{\begin{array}{cc}
			iq_{t}+q_{xx}-2\left(|q|^{2}-1\right)q =0, \\
			q(x,0)=q_{0}(x), \\
   \lim_{x\rightarrow\pm\infty}q_{0}(x)=\pm 1.
		\end{array}\right.
\end{equation}
\subsection{Background and previous results.}
 We first mention the following global-wellposedness result from \cite{G}:
\begin{theorem}
\label{thm:wp-1}
For any $q_0\in E=\left\{f: f \in H_{\mathrm{loc}}^1, \nabla f \in L^2, 1-|f|^2 \in L^2\right\}$,  there exists a unique $w\in  \mathcal{C}\left(\mathbb{R}, H^1\left(\mathbb{R}^n\right)\right)$ such that $q=q_0+w$ solves \eqref{eq:NLS}.
\end{theorem}
Recently, Recently continuous families of conserved energies have been constructed for the defocussing NLS in \cite{KL} and the following global-wellposedness result in low regularity has been obtained:
\begin{theorem}
    Equation \eqref{eq:NLS} is globally well-posed in $$X^0=\left\{q \in L_{\mathrm{loc}}^2(\mathbb{R}):|q|^2-1 \in H^{-1}(\mathbb{R}), \quad \partial_x q \in H^{-1}(\mathbb{R})\right\} / \mathbb{S}^1.$$
\end{theorem}
As is given in \cite{KS}, Equation \eqref{eq:NLS} describes the propagation of waves through a condensate of constant density and is physically relevant in nonlinear optics , fluid mechanics and Bose–Einstein condensation. Comparing with the defocussing NLS that has vanishing boundary at $\pm \infty$ \cite{DZ03}, rather than purely scattering solutions, solutions to \eqref{eq:NLS} exhibit richer dynamics and possess non-trivial solitary waves, namely, the dark solitons.
A rich list of literature has been devoted to study various stability properties of dark/black solitons of \eqref{eq:NLS} of the form:
\begin{equation}
    v_c(x) \equiv \sqrt{\frac{2-c^2}{2}} \tanh \left(\frac{\sqrt{2-c^2}}{2} x\right)-i \frac{c}{\sqrt{2}} 
\end{equation}
and $v_c(x)$ becomes the black soliton $v(x)=\tanh(x)$.
In \cite{BGSA}, the orbital stability of the dark soliton has been proven based on the variation principle. The orbital stability of the black soliton is later proven in \cite{BGSA2}. Concerning the asymptotic stability, the authors of \cite{BGS} and \cite{GS} established the asymptotic stability of dark soliton and black soliton in the energy space respectively.  Also in \cite{BGS14},  the authors studied stability properties of a
sum of dark solitons when their speeds are mutually distinct and distinct from zero
and when the solitons are initially well-separated and spatially ordered according to their speeds.

In a seminal paper \cite{ZS}, the inverse scattering transform (IST) for the NLS equation is developed. Using IST, orbital stability of dark solitons is proven for sufficiently smooth
and decaying perturbations in \cite{GZ}. In \cite{V}-\cite{V2}, assuming \textit{Schwartz} class initial data, using IST and the celebrated \textit{Deift-Zhou} nonlinear steepest descent method \cite{DZ93}, the author computed
both the leading and first correction terms in the asymptotic expansion of the solution $q(x,t)$ and the \textit{partial mass} $\int_{ \pm \infty}^x\left(1-|q(x, t)|^2\right) d x $ in the presence of multi-solitons. 

More recently, in \cite{CJ}, using the $\overline{\partial}$ generalization of the nonlinear steepest descent method, see \cite{DMM18} for detailed expository, the authors derive the leading
order approximation and error bound to the solution for large times in the soliton region of space-time $|x| < t$ for a
 class of initial data whose difference from the non-vanishing background
possesses a fixed number of finite moments and derivatives. And as a corollary an asymptotic stability result for
initial data that are sufficiently close to the $N$-dark solitons follows. More explicitly, the authors assume
\begin{equation}
    q_0-\tanh x\in H^{4,4}(\bbR)
\end{equation}
where
 \begin{equation}
     H^{m,n}(\mathbb{R})=\lbrace f(x): f\in H^m(\bbR)\cap L^{2,n}(\bbR) \rbrace.
 \end{equation}
\subsection{Main results.} In the current paper, we aim to investigate the large time behavior of the partial mass of \eqref{eq:NLS} through the IST method:
 \begin{equation}
     \lim_{t\to\infty} \int_{I} \left(1-|q(x,t)|^2 \right) dx
 \end{equation}
 where $I\subset\mathbb{R}$ is a compact set.
  To facilitate our understanding of the main theorem, for $m,n,k\geq 1$, we define
	\begin{equation}
		\label{space:initial}
		\mathcal{I}^{m,n}_k:= \left\{\begin{array}{cc}
			f(x)-\tanh(x)\in L^{2,m}, \\
			f^{(k)}\in L^{2,n}
		\end{array}\right.,
	\end{equation}
	where 
	$$\|f\|_{L^{2,m}}:=\|(1+|\cdot|)^m f(\cdot)\|_{L^2}.$$
 We may call the function space $\mathcal{I}_k^{m,n}$ \textit{weighted energy space} where the conserved \textit{Ginzburg-Landau} energy of \eqref{eq:NLS} is given by
 \begin{equation}
     \label{E:GL}
     E(t):=\int_{\mathbb {R}} |q_x(x,t)|^2+\left( 1-|q(x,t)|^2\right)^2dx=E(0).
 \end{equation}
 \begin{theorem}
 \label{thm: asym}
     Given $q_0 \in \mathcal{I}^{3,1}_1$, the partial mass of the defocusing NLS equation \eqref{eq:NLS} admit the following large-$t$ asymptotics:
     \begin{align}
         \int^{+\infty}_x \left(|q(s,t)|^2-1\right)ds &=\sum^{N}_{k=1}i\bar{z_k}\left[\text{sol}\left(x-x_k,t;z_k\right)-1\right]-\frac{1}{2\pi}\int_0^{\infty}\log\left(1-|r(z)|^2\right)ds+\mathcal{O}\left(t^{-1}\right).
     \end{align}
     and
     \begin{align}
          \int^x_{-\infty}\left(|q(s,t)|^2-1\right)ds=&-\sum^{N}_{k=1}i\bar{z_k}\left[\text{sol}\left(x-x_k,t;z_k\right)-1\right]-\frac{1}{2\pi}\int^0_{-\infty}\log\left(1-|r(z)|^2\right)ds+\mathcal{O}\left(t^{-1}\right) \\
          \nonumber
    &-2 \sum_{k=1}^N \sin \left(\arg z_k\right).
     \end{align}
 \end{theorem}
 The explicit form of $sol (x-x_k, t; z_k)$ and $x_k$ are given in \eqref{soliton-tanh} and \eqref{xk} respectively.
 \begin{remark}
     We can reproduce exactly the same asymptotic formula for $q(x,t)$ as in \cite[Theorem 1.1]{CJ}. We do not state the explicit formulas here for brevity.
 \end{remark}
 Our main result is the following:
 \begin{theorem}
\label{thm: main}
      Given $q_0 \in \mathcal{I}^{2,0}_{0}$, in the absence of dark soliton, for any fixed $x_1, x_2\in \mathbb{R}$, we have that 
      \begin{equation}
      \label{lim:main}
           \lim_{t\to\infty} \int_{x_1}^{x_2} \left(1-|q(x,t)|^2 \right) dx=0.
      \end{equation}
 \end{theorem}
The purpose of this paper is two-fold:
\begin{itemize}
    \item[1.] We will re-investigate the regularity properties of the direct scattering map. We show that even in the presence of singularities at $\lbrace 0, \pm 1 \rbrace$, we can still make use of the \textit{Fourier}-type argument developed in \cite{Zhou98}. The regularity and decay assumptions on the reflection coefficient necessary for the employment of the nonlinear steepest descent can be deduced from initial condition belonging to function space that is strictly larger than $\Sigma_4$ introduced by \cite{CJ}. Thus we reproduce the $N$-soliton asymptotic stability result of \cite{CJ} using weaker metric.
    \item[2.] The quantity given by \eqref{lim:main} is physically relevant in the sense that it implies that along any finite interval $\mathcal{I}\subset \bbR$, in the absence of black soliton, the density $|q(x,t)|^2$ is arbitrarily close to $1$  as $t\to \infty$. This rules out the possibility of zeros of $|q(x,t)|^2$  except on a set of measure zero. If we further assume that the initial condition $q_0\in \mathcal{I}^{2,0}_1$, then by Theorem \ref{thm:wp-1}, the solution is continuous thus Theorem \ref{thm: main} rules out the possibility of zeros on any bounded set. Physically $|q(x,t)|^2=0$ corresponds to vacuum. And the presence of vacuum prevents the construction and application of the \textit{ Madelung} transform. For more discussions we refer the reader to \cite{CDS}. We establish this result through a \textit{Riemann-Lebsgue} type argument without using the nonlinear steepest descent method thus further relax the decay assumption and remove regularity assumption on the initial condition. The first author developed this argument in a previous joint paper \cite{CLL} to study the asymptotic stability of \textit{sine-Gordon} kinks in some weighted energy space. The key ingredient of the proof is a uniform bound on the norm of the resolvent operator obtained from a continuity-compactness argument first developed in \cite{P16}.
\end{itemize}

	\subsection{List of notations} We give a list of notations for the reader's convenience:\\\\
	\begin{center}
		\begin{tabular}{|c|c|}
			\hline
			Notation     &Summary   \\
			\hline
			$x$, $y$ & The space variable;\\
			\hline
			$\langle x \rangle$ & \tabincell{c}{$\sqrt{1+|x|^2}$}\\
			\hline
			$\zeta(z)$ & The spectral parameter of the \textit{Lax} pair;\\
   \hline
   $\mathcal{F}(f)$ & the \textit{Fourier} transform of $f$;\\
			\hline
			$\|f\|_{L^p_x\left(X;L^q_y(Y)\right)} $& $\left(\int_X\left(\int_Y|f(x,y)|^q dy\right)^{\frac {p}{q}}dx\right)^{\frac {1}{p}}$;\\
			\hline
			$I$ &  $2\times 2$ identity matrix;\\
			\hline
			$\textbf{1}$ & The identity operator;\\
			\hline
			$\psi^{(\pm)}_{1,2}$     &\textit{Jost} solutions of the  \textit{Lax } pair;\\
			\hline
			$m^{(\pm)}_{1,2}$     &Normalized solution of the  \textit{Lax} pair;\\
			\hline
			\tabincell{c}{$S(z)=\left(\begin{array}{cc}
			a(z) & \overline{b(z)} \\
			b(z) & \overline{a(z)}
		\end{array}\right)$} &\tabincell{c}{ The scattering matrix };\\
  \hline
  $z_k$ & Simple zeros of $a(z)$ in $\mathbb{C}^+$;\\
			\hline
   $M(z; x,t)$ & Solution to the \textit{Riemann-Hilbert } problem (RHP);\\
   \hline
$M_\pm(z; x,t)$ & Boundary values for the solutions to RHP;\\
   \hline 
   $\mu(z; x,t)$ & The \textit{Beals-Coifman} solution associated to the RHP\\
   \hline
   $C_\pm f$ & $\lim_{z\to \Sigma_\pm} \int \frac{f(s)}{s-z}\frac{ds}{2\pi i}$ \\
   \hline
		\end{tabular}
	\end{center}
	

	\section{Direct and Inverse scattering}
	\subsection{The direct scattering}
	Following the formalism in \cite{CJ}, the defocusing nonlinear Schr\"odinger (NLS) equation \eqref{eq:NLS} is the compatibility condition for the following Lax pair:
	\begin{equation}
		\label{eq:lax}
		\begin{array}{c}
			\psi_x=\mathcal{L}\psi \\
			i\psi_t=\mathcal{B}\psi
		\end{array}
	\end{equation}
	in the sense that
	\begin{equation}
		i(i\mathcal{L}_{t}-\mathcal{B}_{x}+[\mathcal{L}, \mathcal{B}])=0.
	\end{equation}
	The $\mathcal{L}$ and $\mathcal{B}$ operators are given by
	\begin{subequations}
		\begin{align}
			\mathcal{L}&=i\sigma_3(Q-\lambda(z)),
			\\
			\mathcal{B}&=-2i\lambda(z)\mathcal{L}-(Q^2-I)\sigma_3+iQ_x	
		\end{align}
	\end{subequations}
where
    \begin{equation}\nonumber
        Q=Q(x,t)=\left(\begin{array}{cc}
		0      &    \overline{q(x,t)}\\
		q(x,t) &    0
		\end{array}\right),
		\qquad\lambda(z)=\frac{1}{2}(z+z^{-1})
	\end{equation}
	and $\sigma_3$ is the third \textit{Pauli} matrix:
 \begin{equation}
    \sigma_1=\twomat{0}{1}{1}{0},\qquad
 \sigma_2=\twomat{0}{-i}{i}{0},\qquad
	\sigma_3=\twomat{1}{0}{0}{-1}.
 \end{equation}
 
	We now define the simultaneous solutions ($2\times 2 $ matrix-valued \textit{Jost functions})
	$$\psi^{(\pm)}=\left[ \psi_{1}^{(\pm)}, \psi_{2}^{(\pm)} \right]$$
	to the Lax pair \eqref{eq:lax} in the column-wise form and prescribe the boundary conditions at $\pm\infty$:
	\begin{subequations}
		\begin{align}
			\psi_{1}^{(-)}(x;z) &\sim\left(\begin{array}{l}
				1 \\
			    -z^{-1}
			\end{array}\right) e^{-ix\zeta(z)}, \quad \psi_{2}^{(-)}(x ;z) \sim\left(\begin{array}{l}
				-z^{-1} \\
				1
			\end{array}\right) e^{ix\zeta(z) } \quad \text {as } x \rightarrow-\infty,
			\\
			\psi_{1}^{(+)}(x ; z) &\sim\left(\begin{array}{l}
				1 \\
				z^{-1}
			\end{array}\right) e^{-ix\zeta(z)}, \quad \psi_{2}^{(+)}(x ;z) \sim\left(\begin{array}{l}
				z^{-1} \\
				1
			\end{array}\right) e^{ix\zeta(z) } \quad \text {as } x \rightarrow+\infty,
		\end{align}
	\end{subequations}
	where 
	\begin{equation}
		\zeta(z)=\frac{1}{2}(z-z^{-1}).
	\end{equation}
	We further normalize the \textit{Jost functions} at $x=\pm\infty$:
	\begin{subequations}
		\begin{align}
			\label{jost; normal-1}
			m_{1}^{(\pm)}(x;z)&=\psi_{1}^{(\pm)}(x;z) e^{ix\zeta(z) },\\
			\label{jost; normal-2}
			m_{2}^{(\pm)}(x;z)&=\psi_{2}^{(\pm)}(x;z) e^{-ix\zeta(z) }
		\end{align}
	\end{subequations}
	such that
	\begin{subequations}
		\begin{align}
			\lim _{x \rightarrow \pm \infty} m_{1}^{(\pm)}(x;z)&=(1,\pm z^{-1})^{\T},\\
			\lim _{x \rightarrow \pm \infty} m_{2}^{(\pm)}(x;z)&=(\pm z^{-1},1)^{\T}.
		\end{align}
	\end{subequations}
	It is easy to see from \eqref{eq:lax} that the normalized \textit{Jost functions} \eqref{jost; normal-1}-\eqref{jost; normal-2} satisfy the following \textit{Volterra} integral equations:
	\begin{equation}\label{Volterra; origin}
		\begin{aligned}
			m^{(\pm)}(x;z):&=\left[m_1^{(\pm)}(x;z), m_2^{(\pm)}(x;z)\right]\\
			&=B^{\pm}(z)+\int^x_{\pm\infty}X^\pm(x,z)X^\pm(y,z)^{-1}i\sigma_3(Q(y)\mp\sigma_1)m^{(\pm)}(y;z)e^{i(x-y)\zeta(z)\sigma_3}dy\\
   &=B^{\pm}(z)+\left[ K^\pm_Q m^{(\pm)}\right](x;z)
		\end{aligned}
	\end{equation}
	where
	\begin{equation}
		B^{\pm}(z)=\left(\begin{array}{cc}
			1 & \pm z^{-1}\\
			\pm z^{-1} & 1
		\end{array}\right),\quad X^\pm(x,z)=B^{\pm}(z)e^{-ix\zeta(z)\sigma_3},
	\end{equation}
 and the integral operator (column-wise):
 \begin{align}
 \label{eq:KQ-1}
     \left(K^+_Q \vec{f}_1\right)(x, z)&=\int_{+ \infty}^x\twomat{-i(q-1)\left[-\frac{z}{z^2-1}+\frac{z}{z^2-1}e^{2i\zeta(z)(x-y)}\right]}{i(\bar{q}-1)\left[\frac{z^2}{z^2-1}-\frac{1}{z^2-1}e^{2i\zeta(z)(x-y)}\right]}{-i({q}-1)\left[-\frac{1}{z^2-1}+\frac{z^2}{z^2-1}e^{2i\zeta(z)(x-y)}\right]}{i(\bar{q}-1)\left[\frac{z}{z^2-1}-\frac{z}{z^2-1}e^{2i\zeta(z)(x-y)}\right]}\\
     \nonumber
     &\qquad \times \colvec{f_{11}}{f_{12}} dy,
 \end{align}
 \begin{align}
  \label{eq:KQ-2}
     \left(K^+_Q \vec{f}_2\right)(x, z)&=\int_{+ \infty}^x\twomat{-i(q-1)\left[-\frac{z}{z^2-1}e^{2i\zeta(z)(y-x)}+\frac{z}{z^2-1}\right]}{i(\bar{q}-1)\left[\frac{z^2}{z^2-1}e^{2i\zeta(z)(y-x)}-\frac{1}{z^2-1}\right]}{-i({q}-1)\left[-\frac{1}{z^2-1}e^{2i\zeta(z)(y-x)}+\frac{z^2}{z^2-1}\right]}{i(\bar{q}-1)\left[\frac{z}{z^2-1}e^{2i\zeta(z)(y-x)}-\frac{z}{z^2-1}\right]}\\
     \nonumber
     &\qquad \times \colvec{f_{21}}{f_{22}} dy.
 \end{align}
It is easy to observe from \eqref{eq:KQ-1}-\eqref{eq:KQ-2} that $z=\pm 1$ are removable singularities. We have the following lemmas from \cite{CJ} which will be useful:
	\begin{lemma}
		\label{lemma: 1}
		For every $z \in \mathbb R\backslash \{0,1,-1\}$, if $ q(x)- tanh(x)\in L^1(\mathbb R)$, then \textit{Volterra} integral equations \eqref{Volterra; origin} admits unique solutions $m^{(\pm)}(x;z)$ and $\det m^{(\pm)}=1-z^{-2}$.
	\end{lemma}
        \begin{lemma}
		\label{lem:continuous}
		If $q(x)-tanh(x)\in L^{1,n}(\mathbb{R})$, then $m^{(\pm)}(x;z)\in C^{n-1}\left(\mathbb{R},\mathbb{R}\backslash\{0\}\right) $ and $ 
  \partial_{z}^{n-1}m^{(+)}(x;z)\in L^{\infty}_x([x_0,+\infty);L^{\infty}_z(1-\varepsilon<|z|<1+\epsilon))$ for any preassigned $x_0\in \mathbb{R}$ and $\varepsilon\in(0,1).$
	\end{lemma}
	If $q(x)-\tanh(x)\in L^1(\mathbb{R})$, for every $z \in \mathbb R\backslash \{0,1,-1\}$, there is a matrix $S(z)$, the scattering matrix, with
	\begin{equation}
		\label{eq:S(z)}
		\left[\psi^{(-)}_1,\psi^{(-)}_2\right]=\left[\psi_1^{(+)},\psi_2^{(+)}\right]S(z).
	\end{equation}
	The matrix $S(z)$ takes the form
	\begin{equation}
		S(z)=\left(\begin{array}{cc}
			a(z) & \overline{b(z)} \\
			b(z) & \overline{a(z)}
		\end{array}\right)
	\end{equation}
 where
 \begin{equation}
		\left\{\begin{array}{l}
			a(z)=\frac{\text{det}\left[\psi^{(-)}_1(x;z),\psi^{(+)}_2(x;z)\right]}{1-z^{-2}},\\
			b(z)=\frac{\text{det}\left[\psi^{(+)}_1(x;z),\psi^{(-)}_1(x;z)\right]}{1-z^{-2}}
		\end{array}\right.
	\end{equation}
	and the determinant relation gives
	\begin{equation}
		a(z) \overline{a(z)}-b(z) \overline{b(z)}=1.
	\end{equation}
 We define the reflection coefficient:
	\begin{equation}
		\label{def:r(z)}
		r(z)=\frac{b(z)}{a(z)}.
		\end{equation}
	For $z\in \mathbb{R}\backslash\{0,1,-1\} $, by uniqueness we have
 \begin{equation}
\left\{
\begin{aligned}
\psi_1^{(-)}(x;z) &= \sigma_1\overline{\psi_2^{(-)}(x;z)} ,\\
\psi_2^{(+)}(x;z) &= \sigma_1\overline{\psi_1^{(+)}(x;z)},
\end{aligned}
\right.
\end{equation}

\begin{equation}
\left\{
\begin{aligned}
\psi_1^{(-)}(x;z) &= -z^{-1}\psi_2^{(-)}(x;z^{-1}), \\
\psi_2^{(+)}(x;z) &= z^{-1}\psi_1^{(+)}(x;z^{-1}).
\end{aligned}
\right.
\end{equation}
	
	This leads to the symmetry relation of the entries of $S(z)$:
	\begin{equation}
		-\overline{a(z^{-1})}=a(z), \quad -\overline{b(z^{-1})}=b(z),\quad \overline{r(z^{-1})}=r(z).
	\end{equation}
	
 As stated in Lemma $\ref{lem:continuous}$, provided that $q-\tanh \in L^{1,n}(\mathbb{R})(n\geq1) $, the function $\psi^{(\pm)}_1(x;z), \psi^{(\pm)}_2(x;z)$ is continuous at $z=\pm1 $. In the generic case, we have 
 \begin{equation}
 \label{cond: gen}
	\text{det}[\psi^{(+)}_1(x;\pm1),\psi^{(-)}_1(x;\pm1)]=\mp\text{det}[\psi^{(-)}_1(x;\pm1),\psi^{(+)}_2(x;\pm1)]\neq 0.
    \end{equation} 
    such that  $\lim\limits_{z\rightarrow\pm1}r(z)=\mp1$. In the rest of the paper we will assume \eqref{cond: gen}. 
\subsubsection{The discrete data}
The pair $\psi_1^{-}(x;z_k)$ and $\psi_2^{+}(x;z_k)$ at any zero $z=z_k\in\mathbb{C}^{+}$ of $a(z)$ are linearly dependent. So there exists a constant $\gamma_k\in\mathbb{C}$ such that
\begin{equation}
    \psi_1^{-}(x;z_k)=\gamma_k\psi_2^{+}(x;z_k).
\end{equation}
We have the following lemma from \cite{CJ} characterizing the discrete data:
\begin{lemma}
    If $q(x)-\tanh(x)\in L^{2,3}(\mathbb{R})$, then the zeros of $a(z)$ in $\mathbb{C}^+$ are simple and the number is finite.
\end{lemma}
\begin{remark}
    If we assume the generic condition \eqref{cond: gen}, then $q(x)-\tanh(x)\in L^{2,2}(\mathbb{R})$ will imply the continuity of $\text{det}[\psi^{(-)}_1(x;\pm1),\psi^{(+)}_2(x;\pm1)]$. And by continuity the zeros of $a(z)$ will not accumulate at $z=\pm$ thus  the number of zeros of $a(z)$ in $\mathbb{C}^+$ is also finite.
\end{remark}
We then define the norming constant:
\begin{equation}
    c_k:=\frac{\gamma_k}{a'(z_k)}.
 \end{equation}
	\subsection{Estimations on the reflection coefficient}In this subsection we will prove the following important proposition: 
	\begin{proposition}
		\label{prop:r}
		If $q\in\mathcal{I}^{2,1}_1$ as is given by \eqref{space:initial}, then $r(z)\in H^{1,1}(\mathbb{R})$. Moreover, if $q\in\mathcal{I}^{3,1}_1$, then $r''(z)\in L^2({(1/2, 3/2)\cup (-3/2, -1/2)})$.
		\end{proposition}
	\begin{remark}
 \label{rmk: r}
 Because of singular behavior of $r(z)$ near $0,1,-1$, we will divide the proof into three cases:
		\begin{itemize}
			\item[1.]For $|z|>5/4$, we will directly study problem \eqref{Volterra; origin} in the spirit of  \cite{Zhou98}.
			\item[2.]For $|z|<3/4$, to avoid singularity at the origin, we need the following symmetry: $r(z^{-1})=\overline{r(z)}$. Here we point out that differentiating $\zeta=z/2-1/2z$ with respect to $z$ will result in a second order singularity at $0$, so we consider some weighted $L^2$-norm using change of variable: $z\mapsto 1/z$.
			\item[3.]For $z\in (1/2, 3/2)\cup (-3/2, -1/2)$, we will study the singularities of $r(z)$, $r'(z)$ and $r''(z)$ at $z=\pm1$.
		\end{itemize}
  Our proofs exploits the oscillatory phase function $e^{\pm 2i(x-y)\zeta(z)}$ and the\textit{ Plancherel's} identity. This illustrates the Fourier properties of the direct scattering transform. This allows us to considerabley reduce the regularity and decay assumptions of the initial data when compared with similar results in \cite{CJ}.
 \end{remark}
 The proof of Proposition \ref{prop:r} consists of the following lemmas.
\subsubsection{Away from the origin} 
 \begin{lemma}
    \label{lm:r}
		If $q\in\mathcal{I}^{2,1}_1$, then $r(z)\in L^2$ for $|z|>5/4$.
 \end{lemma}
 \begin{proof}
     Recall that 
     \begin{align} 
b(z) = \frac{\det\left[\psi^{(+)}_1(x;z), \psi^{(-)}_1(x;z)\right]}{1 - z^{-2}} = \frac{\det\left[m^{(+)}_1(x;z), m^{(-)}_1(x;z)\right]}{1 - z^{-2}} e^{-2i\zeta(z)x}.
\end{align}
We observe that 
\begin{align}
\det\left[m^{(+)}_1, m^{(-)}_1\right] &= m^{(+)}_{11}m^{(-)}_{21} - m^{(+)}_{21}m^{(-)}_{11} \\
\nonumber
&= (m^{(+)}_{11} - 1)m^{(-)}_{21} - m^{(+)}_{21}(m^{(-)}_{11} - 1) + m^{(-)}_{21} - m^{(+)}_{21}. 
\end{align}
 By \eqref{def:r(z)} and the fact that $|a(z)|\geq1$,  we only need  to prove that $m^{(\pm)}(x;z)-B^{\pm}(z)\in L^{\infty}_{x}(\mathbb{R}^{\pm};L^2_z(|z|>5/4))$.
	We know that
	\begin{equation}
        \label{cal:m}
		m^{(\pm)}(x;z)-B^{\pm}(z)=(1-K_Q^{\pm})^{-1}K_Q^{\pm}B^{\pm}.
	\end{equation}
	So by standard \textit{Volterra} property (cf.\cite{DZ03} ), we only need to prove $K_Q^\pm B^\pm(x;z)\in L^{\infty}_{x}(\mathbb{R}^{\pm};L^2_z(|z|>5/4))$. Setting $\widetilde{q}=q+\bar{q}-2$
	\begin{equation}
	\begin{aligned}
		\label{eq:K^+_QB^+}
		K_Q^+B^+&=\twomat{ \int^x_{+\infty}\frac{iz}{z^2-1}\widetilde{q}}{\int^x_{+\infty}-\frac{i}{z^2-1}\widetilde{q}}{ \int^x_{+\infty}\frac{i}{z^2-1}\widetilde{q}}{\int^x_{+\infty}-\frac{iz}{z^2-1}\widetilde{q}}\\
		&\quad +\left(\begin{array}{cc}
      \int^x_{+\infty} -\frac{ie^{2i(x-y)\zeta(z)}}{z(z^2-1)}(\bar{q}-1+z^2(q-1))dy &  \int^x_{+\infty} \frac{ie^{2i(y-x)\zeta(z)}}{z^2-1}(z^2(\bar{q}-1)+q-1)dy\\
	   \int^x_{+\infty}-\frac{ie^{2i(x-y)\zeta(z)}}{z^2-1}(\bar{q}-1+z^2(q-1))dy     &   \int^x_{+\infty}\frac{ie^{2i(y-x)\zeta(z)}}{z(z^2-1)}(z^2(\bar{q}-1)+q-1)dy
        \end{array}\right).
	\end{aligned}
    \end{equation}
     We only have to consider
    \begin{equation}
    \label{L^2}
    f(z)=\int^x_{+\infty}e^{2i(x-y)\zeta(z)}(q-1)dy\in L^{\infty}_{x}(\mathbb{R}^+,L^2_z(|z|>5/4))
    \end{equation}
    since other terms are of $O(z^{-1})$ as $|z|\to +\infty$.  To simplify notations, we set $\gamma=z-\frac{1}{z}$, then for $z\in(0,+\infty)$, $\gamma\in (-\infty,+\infty)$, we then let $f(z)=g(\gamma)=g(z-\frac1z)$ and find that
	\begin{equation}
		\label{Variable change}
		\begin{aligned}
			\int_{5/4}^{+\infty}\left|f(z)\right|^2dz \leq\int_0^{+\infty}\left|f(z)\right|^2dz\leq \int_{\mathbb R}\left|g(\gamma)\right|^2\left(\frac 12+\frac {\gamma}{2\sqrt{\gamma^2+4}}\right)d\gamma.
		\end{aligned}
	\end{equation}
 By the\textit{ Plancherel}'s theorem,
	\begin{equation}
		\begin{aligned}
	\|g\|_{ L^{\infty}_{x}(\mathbb{R}^+,L^2_{\gamma}(\mathbb{R}))}&=\sup_{\phi\in C^{\infty}_0, \|\phi\|_{L^{2}}=1 }\left|\int_{\mathbb R}\phi(\gamma)\left(\int^x_{+\infty} e^{i(x-y)\gamma}(q-1)dy\right)d\gamma\right|\\
	&\leq C \sup_{\phi\in C^{\infty}_0, \|\phi\|_{L^{2}}=1}\int^{+\infty}_x\left|\mathcal{F}^{-1}\left[ \phi(\cdot)\right](x-y)\right||q-1|dy\\
	&\leq C \|q-1\|_{L^2(\mathbb{R}^+)}.
    \end{aligned}	
    \end{equation}
 \end{proof}
   \begin{lemma}
    \label{lm:r'}
		If $q\in\mathcal{I}^{2,1}_1$, then $r(z)\in H^1$ for $|z|>5/4$.
 \end{lemma}
  \begin{proof}
    Using the relation \ref{eq:S(z)} and letting $x=0$, we obtain
	\begin{equation}
		\label{eq:S'(z)}
		\dfrac {\partial}{\partial z}S(z)=\left(\dfrac {\partial}{\partial z}m^{(+)}(0;z)^{-1}\right)m^{(-)}(0;z)+m^{(+)}(0;z)^{-1}\left(\dfrac {\partial}{\partial z}m^{(-)}(0;z)\right).
	\end{equation}
	By the standard \textit{Volterra} theory, $\|m^{(\pm)}(x;z)\|_{ L^{\infty}_{x}(\mathbb{R}^{\pm};L^{\infty}_{z}(|z|>5/4))}<\infty$. So we only need to show
	\begin{equation}
		\dfrac {\partial}{\partial z}m^{(\pm)}(0;z) \in L^2_z(|z|>5/4).
	\end{equation}
	Indeed we only need estimates on $m^{(+)}(x;z)$ for $x\geq 0$ and estimates on $m^{(-)}(x;z)$ for $x\leq 0$ follows from symmetry.
	From \eqref{Volterra; origin} we have
\begin{equation}
			\label{def:m}
			\begin{aligned}
				\dfrac {\partial}{\partial z}(m^{(\pm)}-B^{\pm})&=\dfrac {\partial}{\partial z}(K^{\pm}_QB^{\pm})+\left[\dfrac {\partial K^{\pm}_Q}{\partial z}\right](m^{(\pm)}-B^{\pm})+K^{\pm}_Q\dfrac {\partial }{\partial z}\left(m^{(\pm)}-B^{\pm}\right)\\
				&=H_1^{\pm}(x,z)+H_2^{\pm}(x,z)+K^{\pm}_Q\dfrac {\partial }{\partial z}\left(m^{(\pm)}-B^{\pm}\right).
			\end{aligned}
		\end{equation}
	We only need to show 
		$$H_1^{+}(x,z),H_2^{+}(x,z)\in L^{\infty}_{x}(\mathbb{R}^+;L^2_{z}(|z|>5/4)).$$
	For matrix-valued functions, we study one term from each of them and  the proof of the rest are similar. For $H_1^+(x,z)$ we differentiate \eqref{eq:K^+_QB^+} and treat the following term: 
		\begin{equation}
			\begin{aligned}
			h_1^+(x,z)&=\dfrac{\partial}{\partial z}\left(\int^{+\infty}_x\dfrac{iz^2}{z^2-1}e^{2i(x-y)\zeta(z)}(q-1)dy\right)\\
                      &=-\int^{+\infty}_x\dfrac{2iz}{(z^2-1)^2}e^{2i(x-y)\zeta(z)}(q-1)dy-\int^{+\infty}_x\dfrac{z^2+1}{z^2-1}e^{2i(x-y)\zeta(z)}(x-y)(q-1)dy\\
			          &=h^+_{1,1}(x,z)+h^+_{1,2}(x,z).
					\end{aligned}
		\end{equation}
	We will only show that 	$h^+_{1,2}(x,z)\in L^{\infty}_{x}(\mathbb{R}^+;L^2_{z}(|z|>5/4))$. Indeed, we only need to show that
		\begin{equation}
		\tilde{h}^+_{1,2}(x,z)=\int^{+\infty}_xe^{2i(x-y)\zeta(z)}(x-y)(q-1)dy\in L^{\infty}_{x}(\mathbb{R}^+;L^2_{z}(z>M)).
		\end{equation}
	Setting $\gamma=z-\frac1z$
		\begin{equation}
			\begin{aligned}
			\|\tilde{h}^+_{1,2}(x,\gamma)\|_{ L^{\infty}_{x}(\mathbb{R}^+,L^2_{\gamma}(\mathbb{R}))}&=\sup_{\phi\in C^{\infty}_0, \|\phi\|_{L^{2}}=1 }\left|\int_{\mathbb R}\phi(\gamma)\left(\int^{+\infty}_xe^{i(x-y)\gamma}(x-y)(q-1)dy\right)d\gamma\right|\\
	    &\leq C \sup_{\phi\in C^{\infty}_0, \|\phi\|_{L^{2}}=1}\int^{+\infty}_x\left|\mathcal{F}^{-1}\left[ \phi(\cdot)\right](x-y)\right||q-1||y|dy\\
	    &\leq C \|q-1\|_{L^{2,1}(\mathbb{R}^+)}.
		\end{aligned}
	\end{equation}
	We thus show that $H_1^{+}(x,z)\in L^{\infty}_{x}(\mathbb{R}^+;L^2_{z}(|z|>5/4))$. For $H_2^{+}(x,z)$, we only consider the case of the first column:
     	\begin{equation}
 \begin{aligned}
  H_{21}^{+}(x,z)= \left[\dfrac {\partial K^+_Q}{\partial z}\right]\left(m^{(+)}_1-\left(\begin{array}{c}
	1\\
	\frac{1}{z}
  \end{array}\right)\right)=\int_{+\infty}^x\left(\begin{array}{cc}
			 K_{11}^+ & K_{12}^+\\
			 K_{21}^+ & K_{22}^+
		\end{array}\right)\left(m^{(+)}_1-\left(\begin{array}{c}
			1\\
		\frac{1}{z}
	\end{array}\right)\right)dy
   \end{aligned}
	\end{equation}
 where
  \begin{equation}
 \begin{aligned}
		K_{11}^{+}&=-i(q-1)\left(\frac{z^2+1}{(z^2-1)^2}-\frac{z^2+1}{(z^2-1)^2}e^{2i(x-y)\zeta(z)}+\frac{i(z^2+1)(x-y)}{(z^3-z)}e^{2i(x-y)\zeta(z)}\right), \end{aligned}
	\end{equation}
 \begin{equation}
		\begin{aligned}
		K_{12}^+&=i(\bar{q}-1)\left(-\frac{2z}{(z^2-1)^2}+\frac{2z}{(z^2-1)^2}e^{2i(x-y)\zeta(z)}-\frac{i(z^2+1)(x-y)}{(z^2-1)z^2}e^{2i(x-y)\zeta(z)}\right),
      \end{aligned}
	\end{equation}
\begin{equation}
		\begin{aligned}
		K_{21}^+&=-i(q-1)\left(\frac{2z}{(z^2-1)^2}-\frac{2z}{(z^2-1)^2}e^{2i(x-y)\zeta(z)}+\frac{i(z^2+1)(x-y)}{(z^2-1)}e^{2i(x-y)\zeta(z)}\right),
        \end{aligned}
	\end{equation}
\begin{equation}
  \begin{aligned}
		K_{22}^{+}&=i(\bar{q}-1)\left(-\frac{z^2+1}{(z^2-1)^2}+\frac{z^2+1}{(z^2-1)^2}e^{2i(x-y)\zeta(z)}-\frac{i(z^2+1)(x-y)}{(z^3-z)}e^{2i(x-y)\zeta(z)}\right) .\end{aligned}
\end{equation}
	So we need to show that $m^{(+)}-B^+\in L^2_x(\mathbb{R}^+;L^2_z(|z|>5/4))$. We want to establish the following inequality 
	\begin{align}
		\|m^{(+)}-B^+\|_{L^2_x(\mathbb{R}^+;L^2_z(|z|>5/4))}& \leq\|K^+_QB^+\|_{L^2_x(\mathbb{R}^+;L^2_z(|z|>5/4))}\\
  \nonumber
            &\quad +\|K^+_Q(m^{(+)}-B^+)\|_{L^2_x(\mathbb{R}^+;L^2_z(|z|>5/4)}.
	\end{align}
    For $K^+_QB^+$ given by \eqref{eq:K^+_QB^+}, we treat the following terms:
	\begin{equation}
		\begin{aligned}
			h_{2,1}^+(x,z)=\int^x_{+\infty}\frac{iz}{z^2-1}(q-1)dy,
		\end{aligned}
	\end{equation}
	\begin{equation}
		\begin{aligned}
			h_{2,2}^+(x,z)=-\int^x_{+\infty}\frac{iz^2}{z^2-1}e^{2i(x-y)\zeta(z)}(q-1)dy.
		\end{aligned}
	\end{equation}
	Indeed, we only need to study that
	\begin{equation}
		\tilde{h}^+_{2,2}(x,z)=-\int^x_{+\infty}e^{2i(x-y)\zeta(z)}(q-1)dy.
		\end{equation}
	We then apply \textit{Minkowski’s} inequality and \textit{Hardy's} inequality
	\begin{equation}
	\left\|h_{2,1}^+(x,z)\right\|_{L^2_{z}(|z|>5/4)}\leq\int_x^{+\infty}|q-1|dy.
    \end{equation}
	\begin{equation}
		\begin{aligned}
		\left\|\int_x^{+\infty}|q-1|dy\right\|_{L^2_x(R^+)}&=\left(\int_0^{+\infty}\left(\int_x^{+\infty}|q(y)-1|dy\right)^2dx\right)^{\frac12}\\
		&=\left(\int_0^{+\infty}\left(\int_1^{+\infty}x|q(xy)-1|dy\right)^2dx\right)^{\frac12}\\
		&\leq\int_1^{+\infty}\left(\int_0^{+\infty}|x|^2|q(xy)-1|^2dx\right)^{\frac12}dy\\
		&\leq\int_1^{+\infty}y^{-\frac32}\left(\int_0^{+\infty}|xy|^2|q(xy)-1|^2dxy\right)^{\frac12}dy\\
		&\lesssim \|q-1\|_{L^{2,1}(\mathbb{R}^+)}.
	\end{aligned}
	\end{equation}
	Setting $\gamma=z-\frac1z$
	\begin{equation}
		\begin{aligned}
		\left\|\tilde{h}_{2,2}^+(x,\gamma)\right\|_{L^2_{\gamma}(\mathbb{R})}&=\sup_{\phi\in C^{\infty}_0, \|\phi\|_{L^{2}}=1 }\left|\int_{\mathbb R}\phi(\gamma)\left(\int^x_{+\infty}e^{i(x-y)\gamma}(q-1)dy\right)d\gamma\right|\\
	   &\leq C \sup_{\phi\in C^{\infty}_0, \|\phi\|_{L^{2}}=1}\int^{+\infty}_x\left|\mathcal{F}^{-1}\left[ \phi(\cdot)\right](x-y)\right||q-1|dy\\
	    &\leq C \left(\int^{+\infty}_x|q-1|^2dy\right)^{\frac12}.
\end{aligned}
\end{equation}
	\begin{equation}
		\begin{aligned}
			\left\|\left(\int^{+\infty}_x|q-1|^2dy\right)^{\frac12}\right\|_{L^2_x(\mathbb{R}^+)}&=\left(\int_0^{+\infty}\left(\int_x^{+\infty}|q(y)-1|^2dy\right)dx\right)^{\frac12}\\
			&=\left(\int_0^{+\infty}|y||q(y)-1|^2dy\right)^{\frac12} <\infty.
		\end{aligned}
	\end{equation}
	This shows that $r'(z)\in L^2(\mathbb{R}\setminus[-5/4, 5/4])$.
 \end{proof}
 \begin{proposition}
\label{prop:r-low}
  If $q_0\in\mathcal{I}^{2,1}_1$, then $r(z)\in L^{2,1}(\mathbb{R})$.
  \end{proposition}
 \begin{proof}
According to Lemma~\ref{lem:continuous} and Proposition~\ref{prop:r}, we only need to prove that $zr \in L^2( |z| > M )$ for some sufficiently large $M > 1$. Recall that
\begin{align} 
b(z) = \frac{\det\left[\psi^{(+)}_1(x;z), \psi^{(-)}_1(x;z)\right]}{1 - z^{-2}} = \frac{\det\left[m^{(+)}_1(x;z), m^{(-)}_1(x;z)\right]}{1 - z^{-2}} e^{-2i\zeta(z)x}.
\end{align}
We observe that 
\begin{align}
\det\left[m^{(+)}_1, m^{(-)}_1\right] &= m^{(+)}_{11}m^{(-)}_{21} - m^{(+)}_{21}m^{(-)}_{11} \\
\nonumber
&= (m^{(+)}_{11} - 1)m^{(-)}_{21} - m^{(+)}_{21}(m^{(-)}_{11} - 1) + m^{(-)}_{21} - m^{(+)}_{21}\\
\nonumber
&= (m^{(+)}_{11} - 1)m^{(-)}_{21} - m^{(+)}_{21}(m^{(-)}_{11} - 1) + \left(m^{(-)}_{21}-\frac{q(x)}{z} \right)- \left(m^{(+)}_{21}-\frac{q(x)}{z}\right). 
\end{align}
We first point out that from \eqref{eq:K^+_QB^+} we can deduce that
\begin{equation}
    \norm{z(m^{(\pm)}_{11} - 1)m^{(\mp)}_{21}}{L^2(|z|>5/4)}\leq \norm{z(m^{(\pm)}_{11} - 1)}{L^{\infty}(|z|>5/4)}\norm{m^{(\mp)}_{21}}{L^{2}(|z|>5/4}.
\end{equation}
To show $ zm^{(\pm)}_{21}-q \in L^2(|z|>M)$, recall \eqref{eq:K^+_QB^+}, through integration by parts, for the diagonal term, we have that
\begin{equation}
\label{zm}
    \begin{aligned}
        i\int_{+\infty}^x(q(y)-1)e^{2i(x-y)\zeta(z)}dy&=-\frac{z}{z^2-1}\int_{+\infty}^x(q(y)-1)de^{2i(x-y)\zeta(z)}\\
        &=-\frac{z}{z^2-1}(q(x)-1)+\frac{z}{z^2-1}\int_{+\infty}^xq'(y)e^{2i(x-y)\zeta(z)}dy.
    \end{aligned}
\end{equation}
We then conclude that
\begin{equation}
\int_{+\infty}^x (q(y) - 1)e^{2i(x - y)\zeta(z)}\,dy = \mathcal{O}\left(\frac{1}{z}\right). 
\end{equation}
Thus $(\ref{eq:K^+_QB^+})$ becomes
\begin{equation}
K^+_QB^+=\frac{1}{z}\left(\begin{array}{cc}
-i\int^{\infty}_x(q(y)+\overline{q(y)}-2)dy & \overline{q}(x)-1  \\
{q}(x)-1  & i\int^{\infty}_x(q(y)+\overline{q(y)}-2)dy
\end{array}\right)+o(z^{-1}).
\end{equation}
We further deduce that
\begin{equation}
\label{eq: KQ2}
(K^+_Q)^{2}B^+=\frac{1}{z}\left(\begin{array}{cc}
-i\int^{\infty}_x|q(y)-1|^2dy & 0  \\
0 & i\int^{\infty}_x|q(y)-1|^2dy
\end{array}\right)+o(z^{-1}).
\end{equation}
It is important to observe that the off-diagonal entries of $(K^+_Q)^n B^+$ for $n\geq 3$ have the oscillatory factor $e^{\pm 2i(x-y)\zeta}$ which will vanish due to the \textit{Riemann-Lebesgue} lemma. We then observe that as $z\to \infty$,
\begin{equation}
\begin{aligned}
z\left(m^{(+)}(x;z) - B^{+}(z) - \frac{1}{z}\begin{pmatrix}
i\int^{\infty}_x(1 - |q(y)|^2)\,dy & \overline{q}(x) - 1 \\
q(x) - 1 & -i\int^{\infty}_x(1 - |q(y)|^2)\,dy
\end{pmatrix}\right)=\mathcal{O}(1).
\end{aligned}
\end{equation}
Moreover, we observe that the off-diagonal entries of $(K^+_Q)^n B^+$ for $n\geq 3$ have the oscillatory factor $e^{\pm 2i(x-y)\zeta}$  lead to $L^2-$integrability by \textit{Plancherel's} identity. 
In fact, we only need to prove that the following property
\begin{equation}
    \int_{-\infty}^x q'(y)e^{2i(x-y)\zeta(z)}\,dy \in L^2(|z| > M)
\end{equation}
which is a direct consequence of the \textit{Plancherel's} identity.
\end{proof}
 We hereby conclude the first step of proving Proposition \ref{prop:r} listed in Remark \ref{rmk: r}.
  \subsubsection{Near the origin}
 \begin{lemma}
    \label{lm:r-s}
		If $q\in\mathcal{I}^{2,1}_1$, then $r(z)\in L^2$ for $|z|<3/4$.
 \end{lemma}
   \begin{proof}
         From the symmetry $r(z^{-1})=\overline{r}(z)$ and change of variable $z\mapsto \beta=1/z$ it is easy to deduce that
         \begin{equation}
             \int_{|z|<3/4} \left\vert r(z) \right\vert^2 dz=\int_{|\beta|>4/3} \frac{\left\vert r(\beta) \right\vert^2}{\beta^2} d\beta<+\infty.
         \end{equation}
         \end{proof}
         From Lemma~\ref{lem:continuous}, we ascertain that the functions \(\psi^{(\pm)}(x;z)\) are continuous at \(z = \pm 1\). Furthermore, considering that \(\lim_{z \to \pm 1} r(z) = \mp 1\), it follows that \(r(z) \in L^2(\mathbb{R})\).
   
 \begin{lemma}
    \label{lm:r'-s}
		If $q\in\mathcal{I}^{2,1}_1$, then $r(z)\in H^1$ for $|z|<3/4$.
 \end{lemma}
 \begin{proof}
  We make use of the symmetry $r(z^{-1})=\overline{r(z)}$.
	From the symmetry we have
	\begin{equation}
		\int^{3/4}_0|r'(z)|^2dz=\int_{4/3}^{+\infty}|r'(z)|^2z^2dz.
	\end{equation}
	Thus our goal is to show that $r'(z)\in L^{2,1}_z(4/3,+\infty)$.
    From \eqref{eq:S'(z)}, we only show that
	\begin{equation}
		z\dfrac {\partial}{\partial z}m^{(\pm)}(0;z) \in L^2_z(4/3,+\infty).
	\end{equation}	Indeed in \eqref{def:m} (see also \eqref{eq:K^+_QB^+}), we only need to show that
    $$zH_1^{+}(x,z),zH_2^{+}(x,z)\in L^{\infty}_{x}(\mathbb{R}^+;L^2_{z}(4/3,+\infty)).$$
    For $zH_1^{+}(x,z)$, we only consider one term from each of them since the proof of the rest are similar (cf.\eqref{def:m}). Let
\begin{equation}
\begin{aligned}
		zh_1^+(x,z)&=z\dfrac{\partial}{\partial z}\left(\int^{+\infty}_x\dfrac{iz^2}{z^2-1}e^{2i(x-y)\zeta(z)}(q-1)dy\right)\\
				  &=-\int^{+\infty}_x\dfrac{2iz^2}{(z^2-1)^2}e^{2i(x-y)\zeta(z)}(q-1)dy-\int^{+\infty}_x\dfrac{z^3+z}{z^2-1}e^{2i(x-y)\zeta(z)}(x-y)(q-1)dy\\
				  &=zh^+_{1,1}(x,z)+zh^+_{1,2}(x,z).
\end{aligned}
\end{equation}
 We only have to consider the following:
	\begin{equation}
		\begin{aligned}
		\hat{h}_1^+(x,z)&=\int^{+\infty}_xze^{2i(x-y)\zeta(z)}(x-y)(q-1)dy\\
                        &=i\int^{+\infty}_x\frac{\partial e^{2i(x-y)\zeta(z)}}{\partial y}(x-y)(q-1)dy+\int^{+\infty}_xz^{-1}e^{2i(x-y)\zeta(z)}(x-y)(q-1)dy\\
						&=-i\int^{+\infty}_xe^{2i(x-y)\zeta(z)}(x-y)q'dy+i\int^{+\infty}_xe^{2i(x-y)\zeta(z)}(q-1)dy\\
      &\quad +\int^{+\infty}_xz^{-1}e^{2i(x-y)\zeta(z)}(x-y)(q-1)dy\\
                        &=\hat{h}_{1,1}^+(x,z)+\hat{h}_{1,2}^+(x,z)+\hat{h}_{1,3}^+(x,z).
		\end{aligned}
	\end{equation}
	We only consider the first term. Setting $\gamma=z-\frac{1}{z}$, it is easy to deduce
\begin{equation}
	\begin{aligned}
		\left\|\hat{h}_{1,1}^+(x,\gamma)\right\|_{L^{\infty}_x(\mathbb{R}^+;L^2_{\gamma}(\mathbb{R}))}&=\sup_{\phi\in C^{\infty}_0, \|\phi\|_{L^{2}}=1 }\left|\int_{\mathbb R}\phi(\gamma)\left(\int^{+\infty}_xe^{i(x-y)\gamma}(x-y)q'dy\right)d\gamma\right|\\
	    &\leq C \sup_{\phi\in C^{\infty}_0, \|\phi\|_{L^{2}}=1}\int^{+\infty}_x\left|\mathcal{F}^{-1}\left[ \phi(\cdot)\right](x-y)\right||y||q'|dy\\
	    &\leq C \left(\int^{+\infty}_x|y|^2|q'|^2dy\right)^{\frac12}\\
		&\lesssim \|q'\|_{L^{2,1}(\mathbb{R})}.
	\end{aligned}
\end{equation}
For $zH_2^{+}(x,z)$, we  consider 
\begin{equation}
   zH_{21}^{+}(x,z)=z\left[\dfrac {\partial K^+_Q}{\partial z}\right]\left(m^{(+)}_1-\left(\begin{array}{c}
	1\\
	\frac{1}{z}
  \end{array}\right)\right).
\end{equation}
We only consider the following term since other terms contains $1/z$ factors:
\begin{equation}
	\begin{aligned}
	\hat{h}_2^+(x,z)=&\int_{+\infty}^xze^{2i(x-y)\zeta(z)}(x-y)(q-1)(m^{(+)}_{11}-1)dy\\
					=&i\int_{+\infty}^x\frac{\partial e^{2i(x-y)\zeta(z)}}{\partial y}(x-y)(q-1)(m^{(+)}_{11}-1)dy+\int_{+\infty}^xz^{-1}e^{2i(x-y)\zeta(z)}(x-y)(q-1)(m^{(+)}_{11}-1)dy\\
					=&-i\int_{+\infty}^xe^{2i(x-y)\zeta(z)}(x-y)q'(m^{(+)}_{11}-1)dy+i\int_{+\infty}^xe^{2i(x-y)\zeta(z)}(q-1)(m^{(+)}_{11}-1)dy\\
					 &-i\int_{+\infty}^xe^{2i(x-y)\zeta(z)}(x-y)(q-1)\left(\frac{\partial m^{(+)}_{11}}{\partial y}\right)dy+\int_{+\infty}^xz^{-1}e^{2i(x-y)\zeta(z)}(x-y)(q-1)(m^{(+)}_{11}-1)dy\\
					=&\hat{h}_{2,1}^+(x,z)+\hat{h}_{2,2}^+(x,z)+\hat{h}_{2,3}^+(x,z)+\hat{h}_{2,4}^+(x,z).
	\end{aligned}
\end{equation}
Following the proofs of Lemma \ref{lm:r} and Lemma \ref{lm:r'}, we can show that
\begin{equation}
	\begin{aligned}
		\left\|\hat{h}_{2,1}^+(x,z)\right\|_{L^{\infty}_x(\mathbb{R}^+;L^2_z(|z|>4/3))}&\leq\|q'\|_{L^{2,1}(\mathbb{R})}\|m^{(+)}-B^+\|_{L^2_x(\mathbb{R}^+;L^2_z(|z|>4/3))},\\
		\left\|\hat{h}_{2,2}^+(x,z)\right\|_{L^{\infty}_x(\mathbb{R}^+;L^2_z(|z|>4/3))}&\leq\|q-1\|_{L^2(\mathbb{R}^+)}\|m^{(+)}-B^+\|_{L^2_x(\mathbb{R}^+;L^2_z(|z|>4/3))},\\
		\left\|\hat{h}_{2,4}^+(x,z)\right\|_{L^{\infty}_x(\mathbb{R}^+;L^2_z(|z|>4/3))}&\leq\|q-1\|_{L^{2,1}(\mathbb{R}^+)}\|m^{(+)}-B^+\|_{L^2_x(\mathbb{R}^+;L^2_z(|z|>4/3))}.
    \end{aligned}
    \end{equation}
Notice that \eqref{eq:lax} and \eqref{jost; normal-1}, we can obtain that
\begin{equation}
    \frac{\partial}{\partial x}m^{(+)}_1(x;z)=\left(\begin{array}{cc}
       -iz^{-1}  &i\bar{q}  \\
        -iq & iz
    \end{array}\right)m^{(+)}_1(x;z)
\end{equation}
and 
\begin{equation}
    \frac{\partial}{\partial x}m^{(+)}_{11}(x;z)=-iz^{-1}m^{(+)}_{11}(x;z)+i\bar{q}(x)m^{(+)}_{21}(x;z).
\end{equation}
We thus deduce that $\hat{h}_{2,3}^+(x,z)\in L^{\infty}_x(\mathbb{R}^+;L^2_z(|z|>4/3))$. By standard \textit{Volterra} theory, we deduce the conclusion of the lemma.
\end{proof}
We have now completed the second step of Remark \ref{rmk: r}. We then turn to the last part of Remark \ref{rmk: r}.
\subsubsection{Near $\pm 1$}
\begin{lemma}
    \label{lm:r-1}
    If $q\in\mathcal{I}^{2,1}_1$, then $r(z)\in H^1$ for $z\in (1/2, 3/2)\cup (-3/2, -1/2)$.
\end{lemma}
\begin{proof}
 For brevity we only show that $r'(z)\in L^2_z(1/2, 3/2)$. As usual we only need to prove that $\dfrac {\partial}{\partial z}m^{(\pm)}(0;z) \in L^2_z(1/2, 3/2)$.
 
By \eqref{def:m} , we need to prove that 
	$H_1^{+}(x,z),H_2^{+}(x,z)\in L^{\infty}_{x}(\mathbb{R}^+;L^2_{z}(1/2, 3/2))$. For matrix-valued $H_1^+(x,z)$ we treat the entries in the first row since the second row follows from symmetry.
	\begin{equation}
 \label{hbreve-1}
		\begin{aligned}
    \breve{h}^+_1(x,z)&=\int^x_{+\infty}(q-1)\left(-i\frac{z^2+1}{(z^2-1)^2}+i\frac{z^2+1}{(z^2-1)^2}e^{2i(x-y)\zeta(z)}+\frac{(z^2+1)(x-y)}{(z^3-z)}e^{2i(x-y)\zeta(z)}\right)dy\\
	&=\int^x_{+\infty}\left(-i\frac{z^2+1}{(z^2-1)^2}+i\frac{z^2+1}{(z^2-1)^2}e^{2i(x-y)\zeta(z)}+\frac{(z^2+1)(x-y)}{(z^3-z)}e^{2i(x-y)\zeta(z)}\right)d\int^y_{+\infty}(q(s)-1)ds\\
	&=i\frac{z^2+1}{z^2}\int^x_{+\infty}(x-y)e^{2i(x-y)\zeta(z)}\left(\int^y_{+\infty}(q(s)-1)ds\right)dy.
	\end{aligned}
   \end{equation}
   \begin{equation}
   \label{htilde-1}
	\begin{aligned}
    \widetilde{h}^+_1(x,z)&=\int^x_{+\infty}(\bar{q}-1)\left(-i\frac{z^2+1}{(z^2-1)^2}+i\frac{3z^2-1}{z^2(z^2-1)^2}e^{2i(x-y)\zeta(z)}+\frac{(1+z^{-2})(x-y)}{z(z^2-1)}e^{2i(x-y)\zeta(z)}\right)dy\\
	&=\int^x_{+\infty}\left(-i\frac{z^2+1}{(z^2-1)^2}+i\frac{3z^2-1}{z^2(z^2-1)^2}e^{2i(x-y)\zeta(z)}+\frac{(z^{-2}+1)(x-y)}{z^3-z}e^{2i(x-y)\zeta(z)}\right)d\int^y_{+\infty}(\overline{q}(s)-1)ds\\
	&=-\frac{i}{z^2}\int^x_{+\infty}(\bar{q}(y)-1)dy-\int^x_{+\infty}\left(\frac{2}{z^3}-i\frac{z^2+1}{z^4}(x-y)\right)e^{2i(x-y)\zeta(z)}\left(\int^y_{+\infty}(\bar{q}(s)-1)ds\right)dy.
		\end{aligned}
    \end{equation}
    \begin{equation}
     \label{hbreve-2}
		\begin{aligned}
   \breve{h}^+_2(x,z)&=\int^x_{+\infty}(q-1)\left(i\frac{2z}{(z^2-1)^2}-i\frac{2z}{(z^2-1)^2}e^{2i(y-x)\zeta(z)}-\frac{(1+z^{-2})(y-x)}{z^2-1}e^{2i(y-x)\zeta(z)}\right)dy\\
	&=\int^x_{+\infty}\left(i\frac{2z}{(z^2-1)^2}-i\frac{2z}{(z^2-1)^2}e^{2i(y-x)\zeta(z)}-\frac{(1+z^{-2})(y-x)}{z^2-1}e^{2i(y-x)\zeta(z)}\right)d\int^y_{+\infty}(q(s)-1)ds\\
	&=-\int^x_{+\infty}\left(\frac{1}{z^2}-i\frac{z^2+1}{z^3}(y-x)\right)e^{2i(y-x)\zeta(z)})\left(\int^y_{+\infty}(q(s)-1)ds\right)dy.
	\end{aligned}
    \end{equation}
    \begin{equation}
    \label{htilde-2}
		\begin{aligned}
    \widetilde{h}^+_2(x,z)&=\int^x_{+\infty}(\bar{q}-1)\left(i\frac{2z}{(z^2-1)^2}-i\frac{2z}{(z^2-1)^2}e^{2i(y-x)\zeta(z)}-\frac{(z^{2}+1)(y-x)}{z^2-1}e^{2i(y-x)\zeta(z)}\right)dy\\
	&=\int^x_{+\infty}\left(i\frac{2z}{(z^2-1)^2}-i\frac{2z}{(z^2-1)^2}e^{2i(y-x)\zeta(z)}-\frac{(z^{2}+1)(y-x)}{z^2-1}e^{2i(y-x)\zeta(z)}\right)d\int^y_{+\infty}(\bar{q}(s)-1)ds\\
	&=\int^x_{+\infty}\left(1+i\frac{z^2+1}{z}(y-x)\right)e^{2i(y-x)\zeta(z)}\left(\int^y_{+\infty}(\bar{q}(s)-1)ds\right)dy.
	\end{aligned}
    \end{equation}
	Indeed, we only need to show that
		\begin{equation}
  \label{est: hbreve-1}
		\breve{h}^+_{1,1}(x,z)=\int^{+\infty}_x(x-y)e^{2i(x-y)\zeta(z)}\left(\int^{+\infty}_y(q(s)-1)ds\right)dy\in L^{\infty}_{x}(\mathbb{R}^+;L^2_{z}(\mathbb{R}^+)).
		\end{equation}
	Setting $\gamma=z-\frac1z$
	\begin{equation}
		\begin{aligned}
		\|\breve{h}^+_{1,1}(x,\gamma)\|_{ L^{\infty}_{x}(\mathbb{R}^+,L^2_{\gamma}(\mathbb{R}))}&=\sup_{\phi\in C^{\infty}_0, \|\phi\|_{L^{2}}=1 }\left|\int_{\mathbb R}\phi(\gamma)\left(\int^{+\infty}_x(x-y)e^{i\gamma(x-y)}\left(\int^{+\infty}_y(q(s)-1)ds\right)dy\right)d\gamma\right|\\
	  &\leq C \sup_{\phi\in C^{\infty}_0, \|\phi\|_{L^{2}}=1}\int^{+\infty}_x\left|\mathcal{F}^{-1}\left[ \phi(\cdot)\right](x-y)\right|\left|\int^{+\infty}_yy(q(s)-1)ds\right|dy\\
	    &\leq C \left(\int_x^{+\infty}\left|\int_y^{+\infty}y(q(s)-1)ds\right|^2dy\right)^{\frac12}\\
		&=    C \left(\int_x^{+\infty}\left|\int_1^{+\infty}y^2(q(sy)-1)ds\right|^2dy\right)^{\frac12}\\
		&\leq C \int_1^{+\infty}\left(\int_x^{+\infty}y^4|q(sy)-1|^2dy\right)^{\frac12}ds\\
		&=    C \int_1^{+\infty}s^{-\frac52}\left(\int_x^{+\infty}(sy)^4|q(sy)-1|^2dsy\right)^{\frac12}ds\\
		&\lesssim\|q-1\|_{L^{2,2}(\mathbb{R}^+)}.
		\end{aligned}
	\end{equation}
We then turn to $H^+_2(x;z)$ . We first recall the following two inequalities: $\forall x\in\mathbb{R}$
	\begin{align}
		|e^{ix}-1|&\leq|x|,\\
		|e^{ix}-1-ix|&\leq\frac{x^2}{2} .
	\end{align}
	We note that
	\begin{equation}
 \label{eq: KZ}
		\begin{aligned}
		\left[\dfrac {\partial K^+_Q}{\partial z}\right]\left(m^{(+)}_1-\left(\begin{array}{c}
				1\\
				z^{-1}
		\end{array}\right)\right)=\int_{+\infty}^x\left(\begin{array}{cc}
			 K_{11}^+ & K_{12}^+\\
			 K_{21}^+ & K_{22}^+
		\end{array}\right)\left(m^{(+)}_1-\left(\begin{array}{c}
			1\\
			z^{-1}
	\end{array}\right)\right)dy.
        \end{aligned}
	\end{equation}
	We know that
	\begin{equation}
		\begin{aligned}
		|K_{11}^+|&=\left|\frac{z^2+1}{(z^2-1)^2}-\frac{z^2+1}{(z^2-1)^2}e^{2i(x-y)\zeta(z)}+\frac{i(z^2+1)(x-y)}{(z^3-z)}e^{2i(x-y)\zeta(z)}\right|\left|q-1\right|\\
                           &=\frac{z^2+1}{(z^2-1)^2}\left|e^{-2i(x-y)\zeta(z)}-1+2i(x-y)\zeta(z)\right|\left|q-1\right|\\
						   &\leq\frac{z^2+1}{2z^2}(x-y)^2\left|q-1\right|.
        \end{aligned}
	\end{equation}
 and\\
 \begin{equation}
		\begin{aligned}
		|K_{12}^+|&=\left|-\frac{2z}{(z^2-1)^2}+\frac{2z}{(z^2-1)^2}e^{2i(x-y)\zeta(z)}-\frac{i(z^2+1)(x-y)}{(z^2-1)z^2}e^{2i(x-y)\zeta(z)}\right|\left|\bar{q}-1\right|\\
                           &=\left|-\frac{2z}{(z^2-1)^2}e^{2i(x-y)\zeta(z)}\left(e^{-2i(x-y)\zeta(z)}-1+2i(x-y)\zeta(z)\right)+\frac{i(x-y)}{z^2}e^{2i(x-y)\zeta(z)}\right|\left|\bar{q}-1\right|\\
						   &\leq\frac{(x-y)^2}{|z|}\left|\bar{q}-1\right|+\frac{|x-y|}{z^2}\left|\bar{q}-1\right|.
        \end{aligned}
	\end{equation}
So in order to prove that $H_2^{+}(x,z)\in L^{\infty}_{x}(\mathbb{R}^+;L^2_{z}(1/2,3/2))$, we only need to prove that  $m^{(+)}-B^+\in L^2_x(\mathbb{R}^+;L^2_z(1/2,3/2))$ which follows similarly from the proof of Lemma \ref{lm:r'}.
So we show that $r\in H^{1}$.
\end{proof}
The following lemma is necessary in the application of the nonlinear steepest descent:
\begin{lemma}
   If $q\in\mathcal{I}^{3,1}_2$, then $r\in H^1$ and $r''(z)\in L^2\left((1
/2,3/2)\cup(-3/2,-1
/2)\right)$.
  \end{lemma}
  \begin{proof}
      
Direct differentiation gives
  	\begin{equation}
			\begin{aligned}
				\dfrac {\partial^2}{\partial z^2}(m^{(\pm)}-B^{\pm})&=\dfrac {\partial^2}{\partial z^2}(K^{\pm}_QB^{\pm})+\left[\dfrac {\partial^2 K^{\pm}_Q}{\partial z^2}\right](m^{(\pm)}-B^{\pm})+ 2\left[\dfrac {\partial K^{\pm}_Q}{\partial z}\right]\dfrac {\partial }{\partial z}\left(m^{(\pm)}-B^{\pm}\right) \\&+K^{\pm}_Q\dfrac {\partial^2 }{\partial z^2}\left(m^{(\pm)}-B^{\pm}\right).
			\end{aligned}
		\end{equation}
 For the first term, we notice that in \eqref{hbreve-1}-\eqref{htilde-2}, there is no singularities at $\pm 1$. Thus direct computation and arguments similar to \eqref{est: hbreve-1}, we have that 
\begin{equation}
    \norm{\dfrac {\partial^2}{\partial z^2}(K^{+}_QB^{+})}{L^2(z\in (1/2, 3/2))}\lesssim \norm{q-\tanh(x)}{L^{2,3}}.
\end{equation}
For the second term,  from \eqref{eq: KZ}  we deduce that
\begin{align*}
		\left[\dfrac {\partial^2 K^+_Q}{\partial z^2}\right]\left(m^{(+)}_1-\left(\begin{array}{c}
				1\\
				z^{-1}
		\end{array}\right)\right)=\int_{+\infty}^x\left(\begin{array}{cc}
			 \partial K_{11}^+/\partial z &\partial K_{12}^+/\partial z\\
			\partial K_{21}^+/\partial z & \partial K_{22}^+/\partial z
		\end{array}\right)\left(m^{(+)}_1-\left(\begin{array}{c}
			1\\
			z^{-1}
	\end{array}\right)\right)dy
        \end{align*}
and direct calculation leads to 
\begin{align*}
    \frac{\partial K_{11}^+}{\partial z}&=i(q-1)\left[\frac{-2z^3-6z}{(z^2-1)^3}e^{2i\zeta(z)(x-y)}\left(-1+ e^{-2i\zeta(z)(x-y)} +2i\zeta(z)(x-y)+\frac{(z^2-1)^2}{2z^2}(x-y)^2\right)\right.\\
    &\left. \quad +\frac{1}{z^3}(x-y^2)e^{2i\zeta(z)(x-y)}\right].
\end{align*}
Thus, 
\begin{equation}
    \left\vert  \frac{\partial K_{11}^+}{\partial z}\right\vert \leq |q-1|\left(\frac{|x-y|^2}{z^3}+\frac{z^2+3}{3z^2}|x-y|^3\right)
\end{equation}
and the remaining terms $ \frac{\partial K_{12}^+}{\partial z}$, $ \frac{\partial K_{21}^+}{\partial z}$, $ \frac{\partial K_{22}^+}{\partial z}$  will be estimated similarly. Thus as a consequence of the \textit{Cauchy-Schwarz }inequality:
\begin{align*}
    \left\Vert \left[\dfrac {\partial^2 K^{\pm}_Q}{\partial z^2}\right](m^{(\pm)}-B^{\pm}) \right\Vert_{L^\infty_z} &\leq  \left(\int_{+\infty}^x \left\vert \left(\begin{array}{cc}
			 \partial K_{11}^+/\partial z &\partial K_{12}^+/\partial z\\
			\partial K_{21}^+/\partial z & \partial K_{22}^+/\partial z
		\end{array}\right)\right\vert^2 dy \right)^{1/2}\\
   &\quad \times \norm{\left(m^{(+)}_1-\left(\begin{array}{c}
			1\\
			z^{-1}
	\end{array}\right)\right)}{L^2_{(x>0)}}\\
  &\lesssim \norm{q-1}{L^{2.3}_{(x>0)}}.
\end{align*}
We can similarly bound 
\begin{equation}
    \norm{\left[\dfrac {\partial K^{\pm}_Q}{\partial z}\right]\dfrac {\partial }{\partial z}\left(m^{(\pm)}-B^{\pm}\right)}{L^\infty_{z}}\lesssim\norm{q-1}{L^{2.3}_{(x>0)}}.
\end{equation}
By Lemma \ref{lm:r-1}, the proof of this lemma follows.
\end{proof}
We have concluded the proof of Proposition \ref{prop:r}.
We end the section with the following remark which will be useful when we solve the inverse scattering problem:
 \begin{remark}
    \label{remark:rlow}
       In the proof of Lemma \ref{lm:r} and Lemma \ref{lm:r-s}, we actually only need $q_0\in L^2$. Indeed if we assume $q_0\in \mathcal{I}^{2,0}_0=L^{2,2}$ which is a subset of $L^{1,1}$, then by Lemma \ref{lem:continuous}, we can show that $r(z)\in \mathcal{C}(\bbR)$. Moreover, from \eqref{eq:K^+_QB^+} it is easy to deduce from \textit{Riemann-Lebsgue} lemma that $\lim_{z\to \pm \infty}r(z)=0$. Also, from \eqref{eq:K^+_QB^+} and the fact that $\overline{r(z^{-1})}=r(z)$ we can deduce
       \begin{equation}
       \label{cor:sing}
 \lim_{z\to 0 }\left\vert\frac{r(z)}{z}\right\vert=\lim_{z\to \pm \infty} \left\vert z\overline{r(z)}\right\vert<+\infty.
       \end{equation}
    \end{remark}
\subsection{The inverse scattering}
 Standard inverse scattering theory implies that $r(z)$ and $c_k$ have linear time evolution:
\begin{equation}
    r(z,t)=r(z)e^{-it(z^2-z^{-2})},\quad c_k(t)=c_k e^{-it(z^2-z^{-2})}.
\end{equation}
\begin{definition}
\label{df:sd}
We define two sets of \textit{scattering data}:
\begin{align}
    \mathcal{S}_1&=\left\{r(z),\left\{z_k, c_k\right\}_{k=1}^{N}\right\} \in \mathcal{L}_0^{2}(\mathbb{R}) \times \mathbb{C}^{2 N}, \\
    \mathcal{S}_2&=\left\{r(z),\left\{z_k, c_k\right\}_{k=1}^{N}\right\} \in \breve{H}_0^{1,1}(\mathbb{R}) \times \mathbb{C}^{2 N} 
\end{align}
where
\begin{align*}
    \mathcal{L}_0^{2}(\mathbb{R})&=L^2(\bbR)\cap \mathcal{C}(\bbR)\cap\lbrace r(z): |\lim_{z\to 0}r(z)/z|<+\infty \rbrace\cap\lbrace r(z): \lim_{z\to 0}r(z)=0 \rbrace,\\
    \breve{H}^{1,1}_0(\bbR)&=H^{1,1}(\bbR)\cap \lbrace r(z):  |\lim_{z\to 0}r(z)/z|<+\infty \rbrace\cap \lbrace r''(z)\in L^2({(1/2, 3/2)\cup (-3/2, -1/2)})\rbrace.
\end{align*}
and 
\begin{equation}
    |z_k|=1, \quad 0<\arg z_1<\arg z_2<...<\arg z_N<\pi.
\end{equation}
In $\mathcal{S}_2$, we also assume that $z_k\neq i$ for $k=1,2,...N$.
\end{definition}
We characterize the direct scattering map as follows:
\begin{proposition}
\label{prop:L2,2}
    Given $q_0\in \mathcal{I}^{2,0}_0$ and $q_0\in \mathcal{I}^{3,1}_1$, and the generic condition \eqref{cond: gen}, then there are two \textit{Lipschitz} continuous maps:
    \begin{align}
       \mathcal{I}^{2,0}_0 \ni q_0 \mapsto \mathcal{S}_1,\\
       \nonumber
       \mathcal{I}^{3,1}_1 \ni q_0 \mapsto \mathcal{S}_2.
    \end{align}
\end{proposition}
\subsubsection{The Beals-Coifman solution} In this subsection we construct the Beals-Coifman solutions needed for the construction of RHP.  We need to find certain piece-wise analytic matrix functions. An obvious choice is 
\begin{equation}
 \begin{cases}
\left( m^{(-)}_1, m^{(+)}_2\right), \qquad \text{Im} z>0,\\
\\
\left( m^{(+)}_1, m^{(-)}_2\right), \qquad \text{Im} z<0,
\end{cases}
\end{equation}
Here we use the fact that 
$m^{(+)}_1$ and $m^{(-)}_2$ have analytic extension to $\bbC^- $. Similarly $m^{(-)}_1$ 
and $m^{(+)}_2$ have analytic extension to $\bbC^+$.
We want the solution to the RHP normalized as $x\rightarrow +\infty$, so we set
\begin{equation}
\label{BC}
M(z; x)= \begin{cases}
(m^{(-)}_1, m^{(+)}_2)\twomat{a(z)^{-1}}{0}{0}{1}, \qquad \text{Im} z>0,\\
(m^{(+)}_1, m^{(-)}_2)\twomat{1}{0}{0}{\overline{a}(\overline{z})^{-1}}, \qquad \text{Im} z<0.
\end{cases}
\end{equation}
Following the same argument as \cite{CJ}, we conclude that for $z\in \bbR$
\begin{equation}
\label{M+M-}
\left(  \dfrac{ m^{(-)}_1}{a(z)}, m^{(+)}_2\right) \twomat{1}{0}{e^{2ix\zeta(z)}  \dfrac{b(z)}{a(z)} }{1}=\left( m^{(+)}_1,  \dfrac{ m^{(-)}_2}{\overline{a}(\overline{z})}\right) \twomat{1}{-e^{-2ix\zeta(z)}\dfrac{\overline{b}(z)}{\overline{a}(z)} }{0}{1}.
\end{equation}
Setting $M_{\pm}(z; x)=\lim_{\epsilon\to 0^+}M( z\pm i\epsilon, x)$, then $M_{\pm}$ satisfy the following jump condition on $\bbR$:
$$M_+(z; x)=M_-(z; x) \Twomat{1-|r(z)|^2}{-e^{-2ix \zeta(z)} \overline{r(z)} }{e^{2ix \zeta(z)} r(z) }{1}.$$
We also calculate the residue at the pole $z_k$:
\begin{align}
\label{residue1}
\textrm{Res}_{z =z_k}M_{+}(z; x)&=\frac{1}{{a}'(z_k)}\Twomat{m^-_{11}(x,z_k)}{0}{m^-_{21}(x,z_k)}{0}\\
\nonumber
                                    &=\frac{e^{2ix\zeta{(z_i)}}\gamma_k}{{a}'(z_k)}\Twomat{m^+_{12}(x,z_k)}{0}{m^+_{22}(x,z_k)}{0}.
\end{align}
Similarly, at the pole $\overline{z_k}$:
\begin{align}
\label{residue2}
\textrm{Res}_{z =\overline{z_i}}M_{-}(z; x)
=-\frac{e^{-2ix\zeta(\overline{z_k} ) }\overline{\gamma_k}}{\overline{a}'(\overline{z_k})}\Twomat{0}{m^+_{11}(x,\overline{z_k} ) }{0}{m^+_{21}(x,\overline{z_k})}.
\end{align}
We now formulate the following\textit{ Riemann-Hilbert} problem necessary for the reconstruction of solution to \eqref{eq:NLS}:
\begin{problem}
\label{RHP:m}
Given scattering data $\mathcal{S}_j$, $j=1,2$ as in Definition \ref{df:sd}, find a $2\times2$ matrix-valued function $M\left(z;x,t\right)$ such that
\begin{itemize}
    \item[1.] $M$ is meromorphic for $z\in\mathbb{C}\backslash\mathbb{R}$;
    \item[2.] $M(z;x,t)=I+\mathcal{O}(z^{-1})$ as $ z\rightarrow\infty$, $zM(z;x,t)=\sigma_1+\mathcal{O}(z)$ as $ z\rightarrow0$;
    \item[3.] $M_{\pm}(z;x,t):=\lim_{\epsilon\rightarrow 0}M(z\pm i\epsilon;x,t)$ exist for any $z\in\mathbb{R}\backslash\{0\}$ and satisfy the jump relation $M_{+}(z;x,t)=M_{-}(z;x,t)V(z)$ where
    \begin{equation}
        V(z)=\left(\begin{array}{cc}
            1-|r(z)|^2           &  -\overline{r(z)}e^{-2it\theta}\\
            r(z)e^{2it\theta}  &  1
        \end{array}\right)
    \end{equation}
    where
    \begin{equation}
        2it\theta(z;x,t)=2it \left[\xi(z-z^{-1})-\frac{1}{2}(z^2-z^{-2})\right],\quad\xi=\frac{x}{2t};
    \end{equation}
     \item[4.] $M(z;x,t)$ satisfies the symmetries
    \begin{equation}
        \begin{aligned}
            \overline{M(\overline{z};x,t)}&=\sigma_1M(z;x,t)\sigma_1,\\
            M(z^{-1};x,t)&=zM(z;x,t)\sigma_1 ;
            \end{aligned}
    \end{equation}
    \item[5.] $M(z;x,t)$ has simple poles at each  $z_k\in\mathcal{Z}$ and $\overline{z_k}\in\overline{\mathcal{Z}}$, with residues satisfying
    \begin{equation}
    \label{m:Res}
        \begin{aligned}
           \mathop{\Res}\limits _{z=z_k}M(z;x,t)=\lim\limits_{z\rightarrow z_k}M(z;x,t)\left(\begin{array}{cc}
                0 & 0 \\
                c_ke^{2it\theta} & 0
            \end{array}\right),\\
              \mathop{\Res}\limits_{z=\overline{z_k}} M(z;x,t)=\lim\limits_{z\rightarrow \overline{z_k}}M(z;x,t)\left(\begin{array}{cc}
                0 & \overline{c_k}e^{-2it\theta} \\
                0 & 0
            \end{array}\right).
        \end{aligned}
    \end{equation}
\end{itemize}
\end{problem}
\begin{remark}
  For each pole $z_j \in \mathbb{C}^+$, let $\Gamma_j$ be a circle centered at  $z_j$ of sufficiently small radius to be in the open upper half-plane and to be disjoint from all other circles. By doing so we replace the residue conditions  of the Riemann-Hilbert problem with Schwarz invariant jump conditions across closed contours (see Figure \ref{fig:res}). The equivalence of this new RHP on augmented contours with the original one is a well-established result (see \cite{Zhou98} ). The purpose of this replacement is to
\begin{enumerate}
\item make use of the celebrated \textit{vanishing lemma} from \cite[Theorem 9.3]{Zhou89} ;
\item formulate the Beals-Coifman representation of the solution of \eqref{eq:NLS}.
\end{enumerate}
We now rewrite the jump conditions of Problem \ref{RHP:m}:
$M(z; x, t)$ is analytic in $\mathbb{C}\setminus \Sigma$ where 
$$\Sigma= \bbR\cup  \left( \bigcup_{k=1}^{N} \Gamma_k  \right) \cup \left( \bigcup_{k=1}^{N} \Gamma_k^*  \right)$$
 is given in figure \ref{fig:res} below  and has continuous boundary values $M_\pm$ on $\Sigma$
and  $M_\pm$ satisfy
$$ M_+(z; x, t)=M_-(z; x, t)e^{-i\theta(z; x, t)\ad\sigma_3}V(z)$$
where 
\begin{align*}
V(z)=\Twomat{1-|r(z)|^2}{-\overline{r(z)}} {r(z) }{1}, \quad z\in \bbR
\end{align*} 
and 
$$
V(z) = 	\begin{cases}
						\twomat{1}{0}{\dfrac{c_k }{z-z_k}}{1}	&	z\in \Gamma_k, \\
						\\
						\twomat{1}{\dfrac{-\overline{c_k}}{z -\overline{z_k }}}{0}{1}
							&	z \in \Gamma_k^*.
					\end{cases}
$$
\end{remark}
\begin{figure}[h]
\label{fig:res}
\caption{Residue conditions}
\begin{tikzpicture}[scale=0.5]
\draw[ thick] (0,0) -- (-3,0);
\draw[ thick] (-3,0) -- (-5,0);
\draw[thick,->,>=stealth] (0,0) -- (3,0);
\draw[ thick] (3,0) -- (5,0);
\node[above] at 		(2.5,0) {$+$};
\node[below] at 		(2.5,0) {$-$};
\draw [red, fill=red] (-3,2) circle [radius=0.09];
\draw [red, fill=red] (-3,-2) circle [radius=0.09];
\draw [red, fill=red] (-1, 3.464) circle [radius=0.09];
\draw [red, fill=red] (-1,-3.464) circle [radius=0.09];
\node[right] at (5 , 0) {$\bbR$};
\draw[dashed] circle [radius=3.605];
\draw circle [radius=0.1];
 \end{tikzpicture}
    \qquad
    \begin{tikzpicture}[scale=0.5]
\draw[ thick] (0,0) -- (-3,0);
\draw[ thick] (-3,0) -- (-5,0);
\draw[thick,->,>=stealth] (0,0) -- (3,0);
\draw[ thick] (3,0) -- (5,0);
\node[above] at 		(2.5,0) {$+$};
\node[below] at 		(2.5,0) {$-$};
\node[left] at (-3.5 , 2) {$\Gamma_j$};
\node[left] at (-3.5 , -2) {$\Gamma_j^*$};
\node[left] at (-1.5 , 3.464) {$\Gamma_k$};
\node[left] at (-1.5 , -3.464) {$\Gamma_k^*$};
\draw[->,>=stealth] (-0.6,3.464) arc(360:0:0.4);
\draw[->,>=stealth] (-2.6,2) arc(360:0:0.4);
\draw[->,>=stealth] (-0.6,-3.464) arc(0:360:0.4);
\draw[->,>=stealth] (-2.6,-2) arc(0:360:0.4);
\draw [red, fill=red] (-3,2) circle [radius=0.07];
\draw [red, fill=red] (-3,-2) circle [radius=0.07];
\draw [red, fill=red] (-1, 3.464) circle [radius=0.07];
\draw [red, fill=red] (-1,-3.464) circle [radius=0.07];
\node[right] at (5 , 0) {$\bbR$};
\draw[dashed] circle [radius=3.605];
\draw circle [radius=0.1];
\end{tikzpicture}
 \begin{center}
  \begin{tabular} {ccc}
Dark solitons ({\color{red} $\bullet$}) 
\end{tabular}
 \end{center}
\end{figure}
It is well-known that $V(z)$ admits the standard triangular factorization:
\begin{equation}
    V(z)=(I-w_{\theta}^-)^{-1}(I+w_{\theta}^+)
\end{equation}
for triangular matrices $w_{\theta}^{\pm}$:
\begin{equation}
\nonumber
\begin{aligned}
w_{\theta}^+ &= 
    \begin{pmatrix}
        0 & 0 \\
        r(z)e^{2it\theta} & 0
    \end{pmatrix}, & w_{\theta}^- &=
    \begin{pmatrix}
        0 & -\overline{r(z)}e^{-2it\theta} \\
        0 & 0
    \end{pmatrix} & z &\in\mathbb{R},  \\ 
w_{\theta}^+ &= 
    \begin{pmatrix}
        0 & 0 \\
        \frac{c_k e^{2it\theta(z_k)}}{z-z_k} & 0
    \end{pmatrix}, & w_{\theta}^- &=
   \begin{pmatrix}
        0 & 0 \\
        0 & 0
    \end{pmatrix} & z &\in {\Gamma}_k, & \\ 
w_{\theta}^+ &= 
   \begin{pmatrix}
        0 & 0 \\
        0 & 0
    \end{pmatrix}, & w_{\theta}^- &=
    \begin{pmatrix}
        0 & \frac{\overline{c_k} e^{-2it\theta(z_k)}}{z- \overline{z_k}} \\
       0 & 0
    \end{pmatrix}& z &\in {\Gamma}_k^*.
\end{aligned}
\end{equation}
Let 
\begin{equation}
    \mu(z;x,t)=M_{+}(1-w_{\theta}^{+})^{-1}=M_{-}(1-w_{\theta}^{-})^{-1}.
\end{equation}
Then it is well established result since \cite{Zhou89} that if $\mu(z;x,t)$ solves the following \textit{Beals-Coifman} integral equation for $z\in\Sigma\setminus \lbrace 0\rbrace$:
\begin{align}
\label{eq:mu}
     \mu(z;x,t)&=I +\twomat{0}{z^{-1}}{z^{-1}}{0}+C_{w_{\theta}}\mu(z;x,t)\\
     \nonumber
                &=I+\frac{\sigma_1}{z}+C_{+}(\mu w_{\theta}^{-})+C_{-}(\mu w_{\theta}^{+}) ,
\end{align}
the solution of the RHP for $M(z;x,t)$ is 
\begin{equation}
    M(z;x,t)=I+\frac{\sigma_1}{z}+\frac{1}{2\pi i}\int_{\Sigma}\frac{\mu (w_{\theta}^{+}+w_{\theta}^{-})}{s-z}ds\quad z\in\mathbb{C}\backslash\Sigma.
\end{equation}
\begin{proposition}
\label{prop:solve}
    Given sets of scattering data $\mathcal{S}_j$, $j=1,2$ with $\arg z_k\neq \pi/2$, then Problem \ref{RHP:m} is uniquely solvable for $t\gg 1$.
\end{proposition}
\begin{proof}
    We first point out that unlike NLS with vanishing boundary conditions at $\pm \infty$, see \cite{DZ03} for instance, Problem \ref{RHP:m} has singularity at the origin. Yet we deduce that 
   \begin{align}
      \mu(z;x,t)-I-\frac{\sigma_1}{z}&= C_{w_{\theta}}\left[\mu(z;x,t)-I-\frac{\sigma_1}{z}\right]+ C_{w_{\theta}}\left[I+\frac{\sigma_1}{z}\right],\\
      \nonumber
      \Longleftrightarrow\\
      \mu(z;x,t)-I-\frac{\sigma_1}{z}&=\left(\mathbf{1}-C_{w_\theta}\right)^{-1}\left(C_{w_{\theta}}\left[I+\frac{\sigma_1}{z}\right]\right).
   \end{align}
   We hereby reduce the solvability of Problem \ref{RHP:m} to the invertibility of the operator $\mathbf{1}-C_{w_\theta}$. We first observe that the inhomogeneous term $C_{w_{\theta}}\left[I+\frac{\sigma_1}{z}\right]$ has no singularities at $z=\pm 1$ given  \eqref{cor:sing}. Then we construct $\mathbf{1}-C_{\widetilde{w}_{\theta}}$ where 
   \begin{align*}
       \widetilde{w}_{\theta}^\pm \restriction_{\mathbb{R}}={w}_{\theta}^\pm ,\quad {w}_{\theta}^\pm \restriction_{\Gamma\cup\Gamma^*}=\mathbf{0}.
   \end{align*}
It is important to notice that since $r(z)\in \mathcal{L}^2_0$ and $(V(z)+V(z)^\dagger) \restriction_{\bbR}$ is positive definite, thus $\mathbf{1}-C_{\widetilde{w}_{\theta}}$ is \textit{Fredholm} by a rational approximation argument and is of index zero by \cite[Lemma 9.3]{Zhou89}, thus $(\mathbf{1}-C_{\widetilde{w}_{\theta}})^{-1}$ exists. Then by \textit{second resolvent identity}, we have that 
\begin{equation}
   \left(\mathbf{1}-C_{{w}_{\theta}}\right)^{-1}\left[\mathbf{1}-\left(C_{{w}_{\theta}}- C_{\widetilde{w}_{\theta}}\right) (\mathbf{1}-C_{\widetilde{w}_{\theta}})^{-1}\right]=(\mathbf{1}-C_{\widetilde{w}_{\theta}})^{-1}.
\end{equation}
Notice that $C_{\widetilde{w}_\theta}\equiv \mathbf{0}$. 
We also partition the set $\{1,...,N\}$ into the pair of sets
\begin{equation}
    \mathcal{D}^-=\{k:\pi/2 <\arg(z_k)< \pi\}\quad \textrm{and} \quad  \mathcal{D}^+=\{k:0<\arg(z_k)<\pi/2\}.
\end{equation}
Notice that if $k\in \mathcal{D}^- $ for $k=1,2...N$, then ${w}_{\theta}^\pm  \restriction_{\Gamma\cup\Gamma^*}$ decay exponentially as $t\to +\infty$. Thus
\begin{equation}
    \norm{\left(C_{{w}_{\theta}}- C_{\widetilde{w}_{\theta}}\right) (\mathbf{1}-C_{\widetilde{w}_{\theta}})^{-1}}{L^2}\lesssim e^{-t}\norm{\left(\mathbf{1}-C_{\widetilde{w}_{\theta}}\right)^{-1}}{L^2}.
\end{equation}
Thus we can construct $ \left(\mathbf{1}-C_{{w}_{\theta}}\right)^{-1}$ through \textit{Neumann} series. If for some $k=1,2...N$, we have that $1,2,...,k\in \mathcal{D}^+$, then we construct $\widetilde{M}_\pm=T_k(\infty)^{-\sigma_3}M_\pm T_k^{\sigma_3}(z)$ with
\begin{equation}
 T_k(z)=   \prod_{i=1}^k\left(\frac{z-z_i}{zz_i-1}\right)
\end{equation}
to reverse the triangularity of $w_\theta^\pm$ and obtain exponential decay in $t$.
\end{proof}

The potential $q(x,t)$ is recovered by the reconstruction formula
\begin{equation}
    q(x,t)=\lim\limits_{z\rightarrow\infty}z\left(M(z;x,t)-I\right)_{21}
\end{equation}
and the partial mass is obtained through:
\begin{equation}
\label{def:int}
\int_x^{+\infty}\left(|q(y,t)|^2-1\right)dy=-i\lim\limits_{z\rightarrow\infty}z\left(M(z;x,t)-I\right)_{22}.
\end{equation}
We end this section with a proposition about the uniform resolvent bound. 
\begin{proposition}
    Given $r(z)\in H^{1,1}(\mathbb{R})$ with $\norm{r}{H^{1,1}}< \eta$ uniformly bounded, if $(1-C_{w_\theta})^{-1}$ exists, then for some $C_\eta>0$ we have 
    \begin{equation}
    \label{op:bound}
        \sup_{r \in H^{1,1}}\left(\sup_{x, t \in[a,\infty)}\left\|\left(I-C_{w_\theta}\right)^{-1}\right\|_{L^2}\right)<C_\eta.
    \end{equation}
\end{proposition}
\begin{proof}
Recall that the embedding $i: H^{\alpha, \beta}(\mathbb{R}) \hookrightarrow H^{\alpha^{\prime}, \beta^{\prime}}(\mathbb{R})$ is compact for $\alpha>\alpha^{\prime}$ and $\beta>\beta^{\prime}$. Thus we can construct operator $C_{\breve{w}_\theta}$ with $\breve{r}\in H^{1/2+\varepsilon, 1/2+\varepsilon}$ for some $0<\varepsilon<1/2$ and obtain $(\mathbf{1}-C_{\breve{w}_\theta})^{-1}$ a \textit{Fredholm} operator with zero index. The rest of the proof follows from \cite[Appendix B]{JLPS20}.
\end{proof}
\section{The long time asymptotics of partial mass}
\label{sec: formula}
In this section, we reproduce the result of \cite{V2} under weaker initial data assumptions. We do not give the details of the proofs since we employ the same $\overline{\partial}$-nonlinear steepest descent as is given in \cite{CJ}. We solve the RHP Problem \ref{RHP:m} with scattering data $\mathcal{S}_1$ as defined 
 in Definition \ref{df:sd}. To facilitate the statement of our results,
We partition the set $\{1,...,N\}$ into the pair of sets
\begin{equation}
    \mathcal{D}^+_{\xi}=\{j:\Re(z_j)>\xi\}\quad \textbf{and} \quad  \mathcal{D}^-_{\xi}=\{j:\Re(z_j)\leq\xi\}
\end{equation}
where $|\xi|=|x/2t|<1$. Recall the reconstruction formula \eqref{def:int}, then as $t\rightarrow\infty$ with $x-2\Re(z_k)t=\mathcal{O}(1)$, we have that
\begin{align*}
     \int_x^{+\infty}\left(|q(s,t)|^2-1\right)ds&=i\bar{z_k}\left(\text{sol}\left(x-x_k,t;z_k\right)-1\right)-\sum_{i<k}2\Im(z_i)\\
     &\quad -\frac{1}{2\pi}\int_0^{\infty}\log\left(1-|r(s)|^2\right)ds+\mathcal{O}\left(t^{-1}\right)
\end{align*}
     
where 
\begin{equation}
\label{soliton-tanh}
\begin{aligned}
    \text{sol}(x,t;z):=-iz(i\Re( z)+\Im (z)\tanh(\Im (z)(x-2t\Re (z)))),
\end{aligned}
\end{equation}
and
\begin{equation}
\label{xk}
    x_k=\frac{1}{2\Im(z_k)}\left(\log\left(\frac{|c_k|}{2\Im(z_k)}\prod_{\begin{subarray}{l}
j\in\mathcal{D}^+_{\xi}\\
j\neq k
\end{subarray}}\left|\frac{z_k-z_j}{z_kz_j-1}\right|^2\right)-\frac{\Im(z_k)}{\pi}\int_0^{\infty}\frac{\log\left(1-|r(s)|^2\right)}{|s-z_k|^2}ds \right).
\end{equation}
And furthermore
\begin{equation}
\label{q:asy}
    \begin{aligned}
        \int_x^{+\infty}\left(|q(s,t)|^2-1\right)ds=\sum^{N}_{k=1}i\bar{z_k}\left[\text{sol}\left(x-x_k,t;z_k\right)-1\right]-\frac{1}{2\pi}\int_0^{\infty}\log\left(1-|r(z)|^2\right)ds+\mathcal{O}\left(t^{-1}\right).
    \end{aligned}
\end{equation}
Using the following identity: 
\begin{equation}
    \int_{-\infty}^{+\infty}\left(\left|q\left(s, t\right)\right|^2-1\right) \mathrm{d} x^{\prime}=-2 \sum_{k=1}^N \sin \left(\arg z_k\right)-\int_{-\infty}^{+\infty} \ln \left(1-|r(z)|^2\right) \frac{\mathrm{d} z}{2 \pi} ,
\end{equation}
we can easily deduce that 
\begin{align*}
    \int^x_{-\infty}\left(|q(s,t)|^2-1\right)ds=&-\sum^{N}_{k=1}i\bar{z_k}\left[\text{sol}\left(x-x_k,t;z_k\right)-1\right]-\frac{1}{2\pi}\int^0_{-\infty}\log\left(1-|r(z)|^2\right)ds+\mathcal{O}\left(t^{-1}\right) \\
    &-2 \sum_{k=1}^N \sin \left(\arg z_k\right).
\end{align*}

\subsection{The proof of Theorem 1.2}
We now solve the RHP Problem \ref{RHP:m} with scattering data $\mathcal{S}_1=\lbrace r(z),\left\{z_k, c_k\rbrace_{k=1}^{N}\right\} $ as in Definition \ref{df:sd}. Again for brevity during the proof we assume that 
$k\in \mathcal{D}^-$ for $k=1,2...N$ where 
$$\mathcal{D}^{-}=\left\{k: \pi / 2<\arg \left(z_k\right)<\pi\right\}.$$
Otherwise we define a new matrix-valued function $\widetilde{M}(z; x, t) = T_k(\infty)^{-\sigma_3}M(z; x, t)T_k(z)^{\sigma_3}$ where $T_k(z)=\prod\limits_{k\in\mathcal{D}^+}\frac{z-z_k}{zz_k-1}$. For fixed $x\in \mathbb{R}$, we have $\xi=x/2t=o(1)$ as $t\to \infty$, thus
\begin{equation}
\label{phase-decay}
   -\Real i\theta(\overline{z_k})= 4 t \sin (\arg \overline{z_k})(\xi-\cos (\arg\overline{z_k}) )<0.
\end{equation}
From reconstruction formula $\eqref{def:int}$ , we have 
\begin{equation}
\label{exp:q}
\begin{aligned}
 \int_x^{+\infty}\left(|q(s,t)|^2-1\right)ds&=-i\lim\limits_{z\rightarrow\infty}z\left(M(z;x,t)-I\right)_{22}\\
 &=\frac{1}{2\pi }\int_{\mathbb{R}}\mu_{21}\overline{r(z)}e^{-2it\theta}ds+ \sum_{j=1}^{N} {\mu_{21}( \overline{z_k }) {\overline{c_k}} e^{-2i\theta(   \overline{z_k}  )} }.
 \end{aligned}
 \end{equation}
We then take a family of initial data  $\lbrace q_{0,n}\rbrace_{n=1}^\infty$ with $q_{0,n}\in \mathcal{I}^{3,1}_1$ and we can generate the following family of scattering data
$$ {\bigoplus}_{n=1}^\infty \mathcal{S}_{2,n}={\bigoplus}_{n=1}^\infty\left\{ r_n(z),\left\{z_{k,n}, c_{k,n}\right\}_{k=1}^{N}\right\} $$ 
as defined in Definition \ref{df:sd}. We require that
\begin{equation}
    \lim_{n\to +\infty}\norm{q_0-q_{0,n}}{L^{2,2}}=0.
\end{equation}
By Proposition \ref{prop:L2,2} we can deduce that
$$ \lim_{n\to +\infty}\norm{r(z)-r(z)_{n}}{L^{2}}=0$$
and
\begin{equation}
    \lim_{n\to +\infty}\sum_{k=1}^N |z_k-z_{k,n}|+|c_k-c_{k,n}|=0.
\end{equation}
We solve the corresponding RHP and obtain:
\begin{equation}
\label{exp:qn}
\begin{aligned}
    \int_x^{+\infty}\left(|q_n(s,t)|^2-1\right)ds&=\frac{1}{2\pi }\int_{\Sigma}\left(\mu_n (w_{\theta,n}^{+}+w_{\theta,n}^{-})\right)_{22}ds\\
    &=\frac{1}{2\pi }\int_{\mathbb{R}}\mu_{n,21}\overline{r_n(z)}e^{-2it\theta}ds+ \sum_{k=1}^{N} {\mu_{n, 21}( \overline{z_k }) {\overline{c_k}} e^{-2i\theta(   \overline{z_k}  )} }.
    \end{aligned}
\end{equation}
\begin{proof}[Proof of Theorem \ref{thm: main}] We first set $t\gg 1$ to ensure the existence of the resolvent operator $\left(\mathbf{1}-C_{w_{\theta, n}}\right)^{-1}$ using Proposition \ref{prop:solve}. 
    We then compute the difference between \eqref{exp:q} and \eqref{exp:qn}.
    \begin{align*}
       D_n &=\dfrac{1}{2\pi}\int_\bbR \mu_{n, 21}\left(  \overline{r_n(z)}  -\overline{r(z)} \right) e^{-2i\theta} dz+ \dfrac{1}{2\pi} \int_\bbR \left( \mu_{n, 21} - \mu_{21} \right) \overline{r(z)} e^{-2i\theta} dz\\
       &\quad +\sum_{k=1}^{N} \left(\mu_{n, 21}( \overline{z_{k,n} }){\overline{c_{k,n}}} e^{-2i\theta(   \overline{z_{k,n}}  )} -\mu_{21}( \overline{z_k }) {\overline{c_k}} e^{-2i\theta(   \overline{z_k}  )} \right)\\
	&=: D_{n, 1} +D_{n, 2}	+D_{n, 3}	.
	\end{align*}
 We first observe from \eqref{phase-decay} that
 \begin{equation}
 \label{dif-decay}
     \left\vert D_{n, 3}	 \right\vert \lesssim e^{-t}.
 \end{equation}
 The first term can be decomposed as
	\begin{align*}
	D_{n, 1}& =\dfrac{1}{2\pi}\int_\bbR( \mu_{n, 21}-z^{-1})\left(  \overline{r_n(z)}  -\overline{r(z)} \right) e^{-2i\theta} dz+\dfrac{1}{2\pi}\int_\bbR z^{-1}\left(  \overline{r_n(z)}  -\overline{r(z)} \right)  e^{-2i\theta} dz\\
 &=D_{n, 11}+D_{n, 12}.
	\end{align*}
Since we can control the singularity at $z=0$,  
\begin{equation}
\label{Dif-1}
    \left\vert D_{n, 12} \right\vert \lesssim \norm{\overline{r_n(z)}  -\overline{r(z)}}{L^2}.
\end{equation}
For $D_{n, 11}$, we observe that
\begin{align}
\label{Dif-2}
     \left\vert D_{n, 11} \right\vert& \lesssim \norm{\mu_{n, 21}-z^{-1}}{L^2}\norm{\overline{r_n(z)}  -\overline{r(z)}}{L^2}\\
     \nonumber
     &\lesssim \norm{\left(\mathbf{1}-C_{w_{\theta, n}}\right)^{-1}}{L^2}\norm{C_{w_{\theta, n}}\left[I+\frac{\sigma_1}{z}\right]}{L^2}\norm{\overline{r_n(z)}  -\overline{r(z)}}{L^2}\\
     \nonumber
     &\lesssim  \norm{\overline{r_n(z)}  -\overline{r(z)}}{L^2}
\end{align}
where the last inequality follows from the uniform resolvent bound \eqref{op:bound}. 
For $D_{n, 2}$, we  have to bound the difference $\mu_{n, 21} - \mu_{21} $. Instead we study the matrix form:
\begin{align}
\label{dif-1}
  \mu_{n} - \mu&= \left(\mathbf{1}-C_{w_\theta, n}\right)^{-1}\left(C_{w_{\theta, n}}\left[I+\frac{\sigma_1}{z}\right]\right)-\left(\mathbf{1}-C_{w_\theta}\right)^{-1}\left(C_{w_{\theta}}\left[I+\frac{\sigma_1}{z}\right]\right)\\
  \nonumber
  &= \left[\left(\mathbf{1}-C_{w_\theta, n}\right)^{-1}-\left(\mathbf{1}-C_{w_\theta}\right)^{-1}\right]\left(C_{w_{\theta,n}}\left[I+\frac{\sigma_1}{z}\right]\right)\\
  \nonumber
  &\quad +\left(\mathbf{1}-C_{w_\theta}\right)^{-1}\left(C_{w_{\theta,n}}\left[I+\frac{\sigma_1}{z}\right]- C_{w_{\theta}}\left[I+\frac{\sigma_1}{z}\right]\right)\\
  \nonumber
  &:= D_{\mu, n, 1}+D_{\mu, n, 2}.
\end{align}
The uniform bound for $\left(1-C_{w, n}\right)^{-1}$ is obtained in \eqref{op:bound}. To obtain the uniform bound for $\left(1-C_{w}\right)^{-1}$ we repeat the argument in Proposition \ref{prop:solve}. Together with the $L^2$ boundedness of \textit{Cauchy} projection $C_\pm$, it is easy to deduce that 
\begin{equation}
\label{dif-2}
\begin{aligned}
\norm{ D_{\mu, n, 2}}{L^2}  &\leq \norm{\left(\mathbf{1}-C_{w_\theta}\right)^{-1}}{L^2}\norm{\overline{r_n(z)}  -\overline{r(z)}}{L^2}\\
&\lesssim \norm{\overline{r_n(z)}  -\overline{r(z)}}{L^2}.
\end{aligned}
\end{equation}
Making use of the uniform resolvent bound \eqref{op:bound} again and the second resolvent identity:
$$\left\|\left(1-C_w\right)^{-1}-\left(1-C_{w_n}\right)^{-1}\right\|_{L^2} \leq\left\|\left(1-C_w\right)^{-1}\right\|_{L^2}\left\|C_{w_n}-C_w\right\|_{L^2}\left\|\left(1-C_{w, n}\right)^{-1}\right\|_{L^2} $$
we can similarly deduce 
\begin{equation}
\label{dif-3}
    \norm{ D_{\mu, n, 1}}{L^2} \lesssim \norm{\overline{r_n(z)}  -\overline{r(z)}}{L^2}.
\end{equation}
Combining \eqref{dif-1}-\eqref{dif-3}, we obtain
\begin{equation}
\label{dif-4}
     \left\vert D_{n, 2} \right\vert \lesssim \norm{\overline{r_n(z)}  -\overline{r(z)}}{L^2}.
\end{equation}
Given $\varepsilon>0$ we then choose $n\gg 1$ such that $\norm{\overline{r_n(z)}  -\overline{r(z)}}{L^2}<\varepsilon$, then we combine \eqref{dif-decay}, \eqref{Dif-1}-\eqref{Dif-2} and \eqref{dif-4} to deduce that 
\begin{align}
    \int_x^{+\infty}\left(|q(s,t)|^2-1\right)ds &=\int_x^{+\infty}\left(|q_n(s,t)|^2-1\right)ds + \mathcal{O}(e^{-t})+\varepsilon\\
    \nonumber
    &=\sum^{N}_{k=1}i\bar{z_{k,n}}\left[\text{sol}\left(x-x_{k,n},t;z_{k,n}\right)-1\right]-\frac{1}{2\pi}\int_0^{\infty}\log\left(1-|r_n(s)|^2\right)ds\\
    \nonumber
    &\quad + \mathcal{O}(t^{-1})+\varepsilon
\end{align}
where the second equality above follows from \eqref{q:asy} which is obtained from nonlinear steepest descent.
The conclusion of Theorem \ref{thm: main} follows from taking difference and the fact that $\varepsilon$ is arbitrarily small. Notice that the solitons all become constant due to the decay assumption \eqref{phase-decay}.
\end{proof}

\section*{Acknowledgement}
The authors are grateful to Xian Liao from Karlsruhe Institute of Technology for many useful discussions.

\end{document}